\documentclass[a4paper,11pt,english]{smfart}

\usepackage{amssymb}
\usepackage{url}
\usepackage[T1]{fontenc}
\RequirePackage{calrsfs}
\DeclareSymbolFont{rsfscript}{OMS}{rsfs}{m}{b}
\DeclareSymbolFontAlphabet{\mathrsfs}{rsfscript}
\usepackage[utopia,expert]{mathdesign}
\usepackage{lscape}
\usepackage{fancybox}
\usepackage{slashbox,booktabs,amsmath}
\usepackage{makecell}

\usepackage[leftbars]{changebar}

\usepackage{palatino}
\usepackage{rotating}
\usepackage{graphicx}

\usepackage{enumerate}

\usepackage{color}
\definecolor{shadecolor}{gray}{0.90}

\def\proofname{D\'emonstration}

\input xypic
\xyoption{all}
\xyoption{arc}

\makeindex

\def\indexname{Index des notations}
\newtheorem{theo}{Theorem}[section]
\def\thetheo{\thesection.\arabic{theo}}
\newtheorem{prop}[theo]{Proposition}

\newtheorem{lem}[theo]{Lemma}

\newtheorem{coro}[theo]{Corollary}

\renewcommand\theequation{\thetheo}
\def\equat{\refstepcounter{theo}\begin{equation}}
\def\endequat{\end{equation}}

\renewcommand\thetable{\thetheo}

\renewcommand\thesection{\arabic{section}}
\renewcommand\thesubsection{\thesection.\Alph{subsection}}

\def\contentsname{Table des mati\`eres}
\def\listfigurename{Liste des figures}
\def\listtablename{Liste des tables}
\def\appendixname{Appendice}

\def\refname{R\'ef\'erences}

\def\AG{{\mathfrak A}}  \def\aG{{\mathfrak a}}  
    
\def\CG{{\mathfrak C}}    \def\CM{{\mathbb{C}}}
    
  \def\eG{{\mathfrak e}}

\def\IG{{\mathfrak I}}  \def\iG{{\mathfrak i}}

\def\LG{{\mathfrak L}}  \def\lG{{\mathfrak l}}  
  \def\mG{{\mathfrak m}}  
    \def\NM{{\mathbb{N}}}
    
  \def\pG{{\mathfrak p}}  
\def\QG{{\mathfrak Q}}  \def\qG{{\mathfrak q}}  
\def\RG{{\mathfrak R}}  \def\rG{{\mathfrak r}}  \def\RM{{\mathbb{R}}}
\def\SG{{\mathfrak S}}  \def\sG{{\mathfrak s}}  
    
  \def\uG{{\mathfrak u}}

    \def\ZM{{\mathbb{Z}}}

  \def\ab{{\mathbf a}}

  \def\eb{{\mathbf e}}  \def\EC{{\mathcal{E}}}
    \def\FC{{\mathcal{F}}}
\def\Gb{{\mathbf G}}    
\def\Hb{{\mathbf H}}

\def\Lb{{\mathbf L}}    \def\LC{{\mathcal{L}}}
    
    \def\NC{{\mathcal{N}}}

  \def\tb{{\mathbf t}}  
  \def\ub{{\mathbf u}}

\def\Crm{{\mathrm{C}}}    \def\CCB{{\boldsymbol{\mathcal{C}}}}

\def\Mrm{{\mathrm{M}}}    
\def\Nrm{{\mathrm{N}}}

\def\Trm{{\mathrm{T}}}

\def\Zrm{{\mathrm{Z}}}    \def\ZCB{{\boldsymbol{\mathcal{Z}}}}

\def\Gba{{\bar{G}}}          
          
\def\Iba{{\bar{I}}}

\def\Wba{{\bar{W}}}

\def\a{\alpha}

\def\g{\gamma}
\def\G{\Gamma}

\def\D{\Delta}
\def\e{\varepsilon}
\def\ph{\varphi}
\def\l{\lambda}
\def\L{\Lambda}

\def\O{\Omega}
\def\r{\rho}
\def\s{\sigma}

\def\t{\tau}

\def\z{\zeta}

\def\mub{{\boldsymbol{\mu}}}

\DeclareMathOperator{\Ind}{{\mathrm{Ind}}}

\DeclareMathOperator{\Irr}{{\mathrm{Irr}}}
\DeclareMathOperator{\Ker}{{\mathrm{Ker}}}

\DeclareMathOperator{\Mat}{{\mathrm{Mat}}}

\DeclareMathOperator{\rad}{{\mathrm{rad}}}

\DeclareMathOperator{\Res}{{\mathrm{Res}}}

\DeclareMathOperator{\Tr}{{\mathrm{Tr}}}

\def\to{\rightarrow}
\def\longto{\longrightarrow}

\def\DS{\displaystyle}
\def\SS{\scriptstyle}
\def\SSS{\scriptscriptstyle}

\def\finl{~$\blacksquare$}

\def\lexp#1#2{\kern\scriptspace\vphantom{#2}^{#1}\kern-\scriptspace#2}
\def\le{\hspace{0.1em}\mathop{\leqslant}\nolimits\hspace{0.1em}}
\def\ge{\hspace{0.1em}\mathop{\geqslant}\nolimits\hspace{0.1em}}

\mathchardef\inferieur="321E
\mathchardef\superieur="321F

\def\eqna{\begin{eqnarray*}}
\def\endeqna{\end{eqnarray*}}

\def\itemth#1{\item[${\mathrm{(#1)}}$]}

\catcode`\@=11
\long\def\@car#1#2\@nil{#1}
\long\def\@first#1#2{#1}
\long\def\@second#1#2{#2}
\long\def\ifempty#1{\expandafter\ifx\@car#1@\@nil @\@empty
  \expandafter\@first\else\expandafter\@second\fi}
\catcode`\@=12

\DeclareMathOperator{\Ref}{Ref}

\def\boitegrise#1#2{\begin{centerline}{\fcolorbox{black}{shadecolor}{~
    \begin{minipage}[t]{#2}{\vphantom{~}#1\vphantom{$A_{\DS{A_A}}$}}
            \end{minipage}~}}\end{centerline}\medskip}

\def\ve{{\SSS{\vee}}}

\def\surto{\twoheadrightarrow}

\theoremstyle{remark}
\newtheorem{rema}[theo]{Remark}
\newtheorem{notation}{Notation}

\theoremstyle{plain}

\def\Frac{{\mathrm{Frac}}}

\def\BIL{LR}
\def\GAUCHE{L}
\def\CAR{CAR}
\def\FAM{FAM}

\def\euler{{\eb\ub}}
\def\eulerq{{\mathrm{eu}}}

\def\calo{{\Crm\Mrm}}

\def\xyinj{\ar@{^{(}->}}
\def\xysur{\ar@{->>}}

\def\bigrad{{\mathrm{bigr}}}

\DeclareMathOperator{\carac}{{\mathrm{Car}}}
\bigskip\def\unb{{\boldsymbol{1}}}

\def\petitespace{\vphantom{$\DS{\frac{\DS{A^A}}{\DS{A_A}}}$}}

\def\mult{{\mathrm{mult}}}

\makeatletter
\def\hlinewd#1{%
\noalign{\ifnum0=`}\fi\hrule \@height #1 %
\futurelet\reserved@a\@xhline}
\makeatother

\newlength\epaisLigne

\usepackage{multirow}

\def\dotcup{\hskip1mm\dot{\cup}\hskip1mm}

\newcommand{\longiso}{\stackrel{\sim}{\longrightarrow}}

\def\carac{{\mathrm{car}}}

\makeatletter
\def\hlinewd#1{%
\noalign{\ifnum0=`}\fi\hrule \@height #1 %
\futurelet\reserved@a\@xhline}
\makeatother

\usepackage{multirow}

\def\petitespace{\vphantom{$\DS{\frac{\DS{A^A}}{\DS{A_A}}}$}}

\usepackage{lscape}

\def\troncation{{\mathrm{Trunc}}}

\newcommand{\longsurto}{\relbar\joinrel\twoheadrightarrow}
\newcommand{\longinjto}{\lhook\joinrel\longrightarrow}

\begin{document}

%\baselineskip=16pt
%\large\baselineskip=20pt
%\Large\baselineskip=24pt

\title{On the Calogero-Moser space associated \\ with dihedral groups}

\author{{\sc C\'edric Bonnaf\'e}}
\address{
Institut Montpelli\'erain Alexander Grothendieck (CNRS: UMR 5149), 
Universit\'e Montpellier 2,
Case Courrier 051,
Place Eug\`ene Bataillon,
34095 MONTPELLIER Cedex,
FRANCE} 

\makeatletter
\email{cedric.bonnafe@univ-montp2.fr}
\makeatother

\date{\today}

\thanks{The author is partly supported by the ANR 
(Project No ANR-16-CE40-0010-01 GeRepMod).}

\maketitle
\pagestyle{myheadings}

\markboth{\sc C. Bonnaf\'e}{\sc Calogero-Moser space associated with dihedral groups}

%\tableofcontents

Using the geometry of the associated Calogero-Moser space, 
R. Rouquier and the author~\cite{calogero} have attached to any 
finite complex reflection group $W$ 
several notions ({\it Calogero-Moser left, right or two-sided cells}, 
{\it Calogero-Moser cellular characters}), completing the notion of {\it Calogero-Moser families} 
defined by Gordon~\cite{gordon}. If moreover $W$ is a Coxeter group, it is conjectured 
in~\cite[Chapter~15]{calogero} that these notions coincide with the analogous notions defined 
using the Hecke algebra by Kazhdan and Lusztig (or Lusztig in the unequal 
parameters case). 

\def\itembul{\item[$\bullet$]}

In the present paper, we aim to investigate these conjectures 
whenever $W$ is a {\it dihedral group}. Since they are all about the geometry 
of the Calogero-Moser space, we also study some conjectures in~\cite[Chapter~16]{calogero} 
about the fixed point subvariety under the action of a group of roots 
of unity, as well as some other aspects (presentation of the algebra of regular 
functions; cuspidal points as defined by Bellamy~\cite{bellamy cuspidal} and their 
associated Lie algebra). We do not prove all the conjectures but we get at least 
the following results (here, $W$ is a dihedral group of order $2d$, acting on a complex 
vector space $V$ of dimension $2$; we denote by $\ZCB$ its associated Calogero-Moser space 
as in~\cite{calogero} and by $Z$ the algebra of regular functions on $\ZCB$: 
see~\S\ref{sub:calogero} for the definition):
\begin{itemize}
\itembul Calogero-Moser cellular characters and Kazhdan-Lusztig cellular characters 
coincide (this is~\cite[Conjecture~\CAR]{calogero}).

\itembul Calogero-Moser families and Kazhdan-Lusztig families coincide 
(this is~\cite[Conjecture~1.3]{gordon martino}; see also~\cite[Conjecture~FAM]{calogero}). 
This result is not new: it was already proved by Bellamy~\cite{bellamy these}, but we propose 
a slightly different proof, based on the computation of cellular characters.

\itembul We give a presentation of $\CM[V \times V^*]^W$ (this extends 
the results of~\cite{alev} which deal with the case where $d \in \{3,4,6\}$). 
Using~\cite{bonnafe thiel}, we explain how one could derive from this a presentation 
of $Z$: this is done completely only for $d \in \{3,4,6\}$.

\itembul If $d$ is odd, then Calogero-Moser (left, right or two-sided) cells 
coincide with the Kazhdan-Lusztig (left, right or two-sided) cells: this is 
a particular case of~\cite[Conjectures~L~and~LR]{calogero}. For proving this fact, 
we prove that the Galois group defined in~\cite[Chapter~5]{calogero} is equal 
to the symmetric group $\SG_W$ on the set $W$. 

\itembul If $d$ is odd, then we prove~\cite[Conjecture~FIX]{calogero} about the 
fixed point subvariety $\ZCB^{\mub_d}$.
\end{itemize}
We also investigate special cases using calculations with the 
software~ {\tt MAGMA}~\cite{magma}, based on the~{\tt MAGMA} package~{\tt CHAMP} 
developed by Thiel~\cite{thiel}, and a paper in preparation by Thiel and the author~\cite{bonnafe thiel}. 
For instance, we get:
\begin{itemize}
\itembul If $d \in \{3,4,6\}$, then we prove~\cite[Conjecture~FIX]{calogero} 
about the fixed point subvariety $\ZCB^{\mub_m}$ (for any $m$).

\itembul If $d=4$ and the parameters are equal and non-zero 
(respectively $d=6$ and the parameters are generic) and if $\mG$ is a Poisson maximal ideal of $Z$, then 
we prove that the Lie algebra $\mG/\mG^2$ is isomorphic to $\sG\lG_3(\CM)$ 
(respectively $\sG\pG_4(\CM)$). We believe these intriguing examples have their own interest. 
\end{itemize}

\def\carac{{\mathrm{char}}}
\def\rad{{\mathrm{rad}}}
\def\semi{{\mathrm{sem}}}

\bigskip

\noindent{\bf Notation.--- } We set $V=\CM^2$ and we denote by $(x,y)$ the canonical basis of $V$ 
and by $(X,Y)$ the dual basis of $V^*$. We identify $\Gb\Lb_\CM(V)$ with $\Gb\Lb_2(\CM)$. 
We also fix a non-zero natural number $d$, as well as a primitive 
$d$-th root of unity $\z \in \CM^\times$. If $i \in \ZM/d\ZM$, we denote by $\z^i$ 
the element $\z^{i_0}$, where $i_0$ is any representative of $i$ in $\ZM$.

We denote by $\CM[V]$ the algebra of polynomial functions on $V$ (so 
that $\CM[V]=\CM[X,Y]$ is a polynomial ring in two variables) and by $\CM(V)$ its fraction 
field (so that $\CM(V)=\CM(X,Y)$). We will denote by $\otimes$ the tensor product $\otimes_\CM$.

\bigskip

\section{The dihedral group}\label{chapitre:diedral}

\medskip

\subsection{Generators}\label{sub:diedral}
If $i \in \ZM$, we set
$$s_i=\begin{pmatrix} 0 & \z^i \\ \z^{-i} & 0 \end{pmatrix}\qquad\text{and}\qquad
\begin{cases}
s=s_0,\\
t=s_1.\\
\end{cases}
\leqno{\text{and}}$$
Note that $s_i=s_{i+d}$ is a reflection of order $2$ for all $i \in \ZM$ (so that 
we can write $s_i$ for $i \in \ZM/d\ZM$). We set
$$W=\langle s,t \rangle.$$
Then $W$ is a dihedral group of order $2d$, and $(W,\{s,t\})$ is a Coxeter system, 
where
$$s^2=t^2=(st)^d=1.$$
If we need to emphasize the natural number $d$, we will denote by $W_d$ the group $W$. 

\bigskip
% 
% \begin{rema}\label{rem:d=1}
% In Shephard-Todd notation, we have $W=G(d,d,2)$. 
% If $d=1$, then $s=t$ and $W=\langle s \rangle$ has order $2$.
% If $d=2$, then $st=ts$ and $W=\langle s \rangle \times \langle t \rangle$. 
% If $d=3$, $4$ or $6$, then $W$ is a Weyl group of type 
% $A_2$, $B_2$ or $G_2$ respectively.\finl
% \end{rema}
% 
% \bigskip

We set 
$$c=ts=\begin{pmatrix} \z & 0 \\ 0 & \z^{-1} \end{pmatrix},$$
so that the following equalities are easily checked (for all $i$, $j \in \ZM$)
\equat\label{eq:c}
c s_i c^{-1} = s_{i+2}\qquad\text{and}\qquad s_is_j = c^{i-j}.
\endequat
It then follows that 
\equat\label{eq:s-and-t}
\text{\it $s$ and $t$ are conjugate in $W$ if and only if $d$ is odd.}
\endequat
Note that
\equat\label{eq:W}
W=\{c^i~|~i \in \ZM/d\ZM\} \dotcup \{s_i ~|~i \in \ZM/d\ZM\}.
\endequat
The set $\Ref(W)$ of reflections of $W$ is equal to 
\equat\label{eq:ref}
\Ref(W)=\{s_i~|~i \in \ZM/d\ZM\}.
\endequat
Now, let
$$\a_i^\ve = \z^i x-y\qquad\text{and}\qquad \a_i = X - \z^i Y,$$
so that
\equat\label{eq:racines}
s_i(\a_i^\ve)=-\a_i^\ve\qquad\text{and}\qquad s_i(\a_i)= -\a_i.
\endequat
Finally, we fix a primitive $2d$-th root $\xi$ such that $\xi^2=\z$ and we set
$$\t=\begin{pmatrix} 0 & \xi \\ \xi^{-1} & 0 \end{pmatrix}.$$
Then it is readily seen that
\equat\label{eq:tau}
\t s \t^{-1} = t, \qquad \t t \t^{-1} = s\qquad\text{and}\qquad \t^2=1,
\endequat
so that $\t \in \Nrm_{\Gb\Lb_\CM(V)}(W)$. 

\bigskip

\begin{rema}\label{rem:tau}
If $d=2e-1$ is odd, then $\xi=-\z^e$ and so $\t=-s_e$  
induces an inner automorphism of $W$ (the conjugacy by $s_e$). 
If $d$ is even, then $s$ and $t$ 
are not conjugate in $W$ and so $\t$ induces a non-inner automorphism 
of $W$.\finl
\end{rema}

\bigskip

\subsection{Irreducible characters}
We denote by $\unb_W$ the trivial character of $W$ and let $\e : W \to \CM^\times$, $w \mapsto \det(w)$. 
If $d$ is even, then there exist two other linear characters $\e_s$ and $\e_t$ which 
are characterized by the following properties:
$$
\begin{cases}
\e_s(s)=\e_t(t)=-1,\\
\e_s(t)=\e_t(s)=1.
\end{cases}
$$
If $k \in \ZM$, we set 
$$\begin{array}{rccc}
\r_k : & W & \longto & \Gb\Lb_2(\CM)  \\
& s_i & \longmapsto & s_{ki} \\
& c^i & \longmapsto & c^{ki}.
\end{array}
$$
It is easily checked from~(\ref{eq:c}) that $\r_k$ is a morphism of groups (that is, 
a representation of $W$). If $R$ is any $\CM$-algebra, we still denote by 
$\r_k : RW \to \Mat_2(R)$ the morphism of algebras induced by $\r_k$. 
The character afforded by $\r_k$ is denoted by $\chi_k$. The following 
proposition is well-known:

\bigskip

\begin{prop}\label{prop:irr}
Let $k \in \ZM$. Then:
\begin{itemize}
\itemth{a} $\chi_k=\chi_{-k}=\chi_{k+d}$.

\itemth{b} If $d$ is odd and $k \not\equiv 0 \mod d$, then $\chi_k$ is irreducible. 

\itemth{c} If $d$ is even and $k \not\equiv 0$ or $d/2 \mod d$, then $\chi_k$ is irreducible. 

\itemth{d} $\chi_0=\unb_W + \e$ and, if $d$ is even, then $\chi_{d/2}=\e_s + \e_t$.
\end{itemize}
\end{prop}

\bigskip

\begin{coro}\label{coro:irr}
Recall that $\t$ is the element of $\Nrm_{\Gb\Lb_\CM(V)}(W)$ defined in~\S\ref{sub:diedral}. 
\begin{itemize}
\itemth{a} If $d$ is odd, then $|\Irr(W)|=(d+3)/2$ and 
$$\Irr(W)=\{\unb_W,\e\} \dotcup \{\chi_k~|~1 \le k \le (d-1)/2\}.$$
Moreover, $\t$ acts trivially on $\Irr(W)$.
\itemth{b} If $d$ is even, then $|\Irr(W)|=(d+6)/2$ and 
$$\Irr(W)=\{\unb_W,\e,\e_s,\e_t\} \dotcup \{\chi_k~|~1 \le k \le (d-2)/2\}.$$
Moreover, $\lexp{\t}{\chi}=\chi$ if $\chi \in \Irr(W) \setminus \{\e_s,\e_t\}$ 
while $\lexp{\t}{\e_s}=\e_t$.
\end{itemize}
\end{coro}

\bigskip

\subsection{Some fractions in two variables}
We work in the fraction field $\CM(V)=\CM(X,Y)$. If $1 \le k \le d$, then 
\equat\label{eq:impair-1}
\sum_{i \in \ZM/d\ZM} \frac{\z^{ki}}{X - \z^i}=\frac{dX^{k-1}}{X^d-1}.
\endequat
% \bigskip
% 
% If $1 \le r \le d$, then 
% \equat\label{eq:2}
% \sum_{i \in \ZM/d\ZM} \frac{\z^{ri}}{(X - \z^i)^2}=\frac{dX^{r-2}\bigl((d+1-r)X^d+r-1\bigr)}{(X^d-1)^2}.
% \endequat
% \begin{proof}
% This follows from~(\ref{eq:1}) by derivation through the variable $X$. 
% \end{proof}
% 
% \bigskip
% 
% If $1 \le r \le d$, then 
\equat\label{eq:impair-2}
\sum_{i \in \ZM/d\ZM} \frac{\z^{ki}}{X - \z^i Y}=\frac{dX^{k-1}Y^{d-k}}{X^d-Y^d}.
\endequat
\begin{proof}
Let us first prove~(\ref{eq:impair-1}). 
Since $1 \le k \le d$, there exist complex numbers $(\xi_i)_{i \in \ZM/d\ZM}$ such that 
$$\frac{dX^{k-1}}{X^d-1} = \sum_{i \in \ZM/d\ZM} \frac{\xi_i}{X-\z^i}.$$
Then
$$\xi_i = \lim_{z \to \z^i} \frac{dz^{k-1}(z-\z^i)}{z^d-1} = d \z^{(k-1)i} 
\prod_{\substack{\SS{j \in \ZM/d\ZM}\\ \SS{j \neq i}}} (\z^i-\z^j)^{-1}= 
d \z^{(k-1)i-(d-1)i} \prod_{j=1}^{d-1} (1-\z^j)^{-1} = \z^{ki},$$
and~(\ref{eq:impair-1}) is proved.

\medskip

Now,~(\ref{eq:impair-2}) follows easily from~(\ref{eq:impair-1}) by replacing $X$ by $X/Y$.
\end{proof}

% \bigskip
% 
% If $k \in \ZM$, then 
% \equat\label{eq:4}
% \sum_{i \in \ZM/d\ZM} \frac{1}{(X - \z^i Y)(X-\z^{i+k} Y)}=
% \begin{cases}
% \DS{\frac{d X^{d-2}}{X^d-Y^d}} & \text{if $k \not\equiv 0 \mod d$,}\\
% \\
% \DS{\frac{dX^{d-2}(X^d+(d-1)Y^d)}{(X^d-Y^d)^2}} & \text{if $k \equiv 0 \mod d$.}
% \end{cases}
% \endequat
% \begin{proof}
% If $k \equiv 0 \mod d$, this follows easily from~(\ref{eq:2}) by replacing $X$ by $X/Y$ 
% and $r$ by $d$. Assume now that $k \not\equiv 0 \mod d$...
% \end{proof}

\bigskip

If $d=2e$ is even and $1 \le k \le e$, then
\equat\label{eq:pair-1}
\sum_{i \in \ZM/e\ZM} \frac{\z^{2ki}}{X-\z^{2i}Y} = \frac{eX^{k-1}Y^{e-k}}{X^e-Y^e},
\endequat
\equat\label{eq:pair-2}
\sum_{i \in \ZM/e\ZM} \frac{\z^{k(2i+1)}}{X-\z^{2i+1}Y} = -\frac{eX^{k-1}Y^{e-k}}{X^e+Y^e},
\endequat
\equat\label{eq:pair-3}
\sum_{i \in \ZM/e\ZM} \frac{\z^{-(k-1)(2i+1)}}{X-\z^{2i+1}Y} = \frac{eX^{e-k}Y^{k-1}}{X^e+Y^e},
\endequat
\begin{proof}
The equality~(\ref{eq:pair-1}) follows from~(\ref{eq:impair-1}) by replacing $\z$ by $\z^2$ 
and $d$ by $e$. The equality~(\ref{eq:pair-2}) follows from~(\ref{eq:pair-1}) 
by replacing $Y$ by $\z Y$ (note that $\z^e=-1$). Finally, the equality~(\ref{eq:pair-3}) 
follows from~(\ref{eq:pair-2}) by replacing $k$ by $e-k+1$ 
(note that $\z^{-(k-1)(2i+1)}=-(\z^e)^{2i+1}\z^{-(k-1)(2i+1)}= - \z^{(e-k+1)(2i+1)}$). 
\end{proof}

\section{Invariants}

\bigskip

The aim of this section is to describe generators and relations for 
the invariant algebra $\CM[V \times V^*]^W$. Note that such results have been obtained 
if $d \in \{3,4,6\}$ in~\cite{alev}. We set
$$q=xy,\qquad r=x^d+y^d,\qquad Q=XY\qquad \text{et}\qquad R=X^d+Y^d.$$
Then
$$\CM[V]^W=\CM[Q,R]\qquad\text{and}\qquad \CM[V^*]^W=\CM[q,r].$$
We set $P_\bullet=\CM[V]^W \otimes \CM[V^*]^W = \CM[q,r,Q,R] \subset \CM[V \times V^*]^W$. 
If $0 \le i \le d$, we set 
$$\ab_{i,0}=x^{d-i}Y^i + y^{d-i} X^{i}.$$
Note that
$$\ab_{0,0}=r\qquad\text{and}\qquad \ab_{d,0} = R.$$
Finally, let 
$$\euler_0=xX+yY.$$
Then $\ab_{i,0}$, $\euler_0 \in \CM[V \times V^*]^W$. 

We will now describe some relations between these invariants. For this, let 
$\euler_0^{(i)}=(xX)^i+ (yY)^i$. 
Then the $\euler_0^{(i)}$'s belong also to $\CM[V \times V^*]^W$. 
As they will appear in relations between generators of $\CM[V \times V^*]^W$, 
we must explain how to express them as polynomials in $\euler_0$. Firt of all, 
\eqna
\euler_0^i&=&\DS{\sum_{j=0}^i {i \choose j} (xX)^j(yY)^{i-j}} \\
&=& \DS{\sum_{0 \le j < \frac{\SS{i}}{\SS{2}}} {i \choose j} (qQ)^j\bigl((xX)^{i-2j} + (yY)^{i-2j}\bigr)
+}
\begin{cases}
0 & \text{if $i$ is odd,}\\
\DS{{i \choose \frac{\SS{i}}{\SS{2}}}}(qQ)^{i/2} & \text{if $i$ is even.}
\end{cases}
\endeqna
Therefore,
$$\euler_0^i=\sum_{0 \le j < \frac{{i}}{{2}}} {i \choose j} (qQ)^j\,\euler_0^{(i-2j)}
+
\begin{cases}
0 & \text{if $i$ is odd,}\\
\DS{{i \choose \frac{\SS{i}}{\SS{2}}}} (qQ)^{i/2} & \text{if $i$ is even.}
\end{cases}
$$
So, by triangularity of this formula, an easy induction shows that there 
% \equat\label{eq:euler-diedral}
% \ZM[qQ] \oplus 
% \ZM[qQ] \,\euler_0 \oplus \cdots \oplus \ZM[qQ] \,\euler_0^i =
% \ZM[qQ] \oplus 
% \ZM[qQ] \,\euler_0^{(1)} \oplus \cdots \oplus \ZM[qQ] \,\euler_0^{(i)},
% \endequat
% so that we have a recursive way to obtain $\euler_0^{(i)}$ as $\ZM[qQ]$-linear combinations 
% of the $\euler_0^j$'s. More precisely, there 
exists a family of integers 
$(n_{i,j})_{0 \le j \le i/2}$ such that 
\equat\label{eq:changement-base-euler-diedral}
\euler_0^{(i)}=\sum_{0 \le j \le \frac{{i}}{{2}}} n_{i,j} (qQ)^j \euler_0^{i-2j},
\endequat
with $n_{i,0}=1$ for all $i$. 
% D'autre part,
% \equat\label{eq:alpha-diedral}
% \alpb_{i,0}=
% \begin{cases}
% r \ab_{i,0}-q^{d-i} \,\euler_0^{(i)} & \text{si $0 \le i \le d$},\\
% r R - \euler_0^{(d)} & \text{si $i=d$},\\
% R \ab_{i-d,0} - Q^{i-d} \euler_0^{(2d-i)} & \text{si $d \le i \le 2d$}.
% \end{cases}
% \endequat

On the other hand, 
one can check that the following relations hold (for $1 \le i \le j \le d-1$):
%\equat\label{eq:relations-diedral}
$$\begin{cases}
(\Zrm_i^0) \quad \euler_0 \ab_{i,0} = q \ab_{i+1,0} + Q \ab_{i-1,0} \\
(\Zrm_{i,j}^0) \quad \ab_{i,0} \ab_{j,0} =  q^{d-j} Q^i \euler_0^{(j-i)} +
\begin{cases}
r \ab_{i+j,0}-q^{d-i-j} \,\euler_0^{(i+j)} & \text{if $2 \le i+j \le d$},\\
r R - \euler_0^{(d)} & \text{if $i+j=d$},\\
R \ab_{i+j-d,0} - Q^{i+j-d} \euler_0^{(2d-i-j)} & \text{if $d \le i+j \le 2d-2$}.
\end{cases}\\
\end{cases}$$
Using~(\ref{eq:changement-base-euler-diedral}), these last relations can be viewed as relations 
between $q$, $r$, $Q$, $R$, $\euler_0$, $\ab_{1,0}$, $\ab_{2,0}$,\dots, $\ab_{d-1,0}$.

\bigskip

\begin{theo}\label{theo:diedral-invariants}
The algebra of invariants $\CM[V \times V^*]^W$ admits the following presentation:
$$
\begin{cases}
\text{\it Generators:} & q,r,Q,R,\euler_0,\ab_{1,0},\ab_{2,0},\dots,\ab_{d-1,0} \\
\text{\it Relations:} & 
\begin{cases}
(\Zrm_i^0) & \text{\it for $1 \le i \le d-1$,}\\
(\Zrm_{i,j}^0) & \text{\it for $1 \le i \le j \le d-1$.}\\
\end{cases}
\end{cases}
$$
This presentation is minimal, as well by the number of generators as by the number of relations 
(there are $d+4$ generators and $(d+2)(d-1)/2$ relations). Moreover,
$$
\CM[V\times V^*]^W = P_\bullet \oplus 
P_\bullet \,\euler_0 \oplus P_\bullet \,\euler_0^2 \oplus \cdots \oplus P_\bullet \,\euler_0^d 
\oplus P_\bullet \,\ab_{1,0} \oplus P_\bullet \,\ab_{2,0} \oplus \cdots \oplus P_\bullet \,\ab_{d-1,0}.
$$
\end{theo}

\bigskip

\begin{proof}
Let $E^*$ denote the subspace of $\CM[V]$ defined by 
$$E^*= \CM \oplus \Bigl(\bigoplus_{i=1}^{d-1} \bigl(\CM X^i \oplus \CM Y^i\bigr)\Bigr) \oplus \CM (X^d-Y^d).$$
Then $E^*$ is a graded sub-$\CM[W]$-module of $\CM[V]$. 
If $0 \le i \le d$, let $E_i^*$ denote the homogeneous component of degree $i$ of $E^*$. 
Whenever $1 \le i \le d-1$, it affords $\chi_i$ for character whereas $E_0^*$ and $E_d^*$ 
afford respectively $\unb_W$ and $\e$ for characters. Similarly, we define 
$$E= \CM \oplus \Bigl(\bigoplus_{i=1}^{d-1} \bigl(\CM x^i \oplus \CM y^i\bigr)\Bigr) \oplus \CM (x^d-y^d).$$
Then $E$ is a graded sub-$\CM[W]$-module of $\CM[V^*]$. 
If $0 \le i \le d$, let $E_i$ denote the homogeneous component of degree $i$ of $E$. 
Whenever $1 \le i \le d-1$, it affords $\chi_i$ for character whereas $E_0$ and $E_d$ 
afford respectively $\unb_W$ and $\e$ for characters. 
Moreover, the morphism of $\CM[V]^W$-modules $\CM[V]^W \otimes E^* \to \CM[V]$ induced 
by the multiplication is a $W$-equivariant isomorphism and the morphism of 
$\CM[V^*]^W$-module $\CM[V^*]^W \otimes E \to \CM[V^*]$ 
induced by the multiplication is a $W$-equivariant isomorphism. Consequently, 
$$\CM[V \times V^*]^W = P_\bullet \otimes (E^* \otimes E)^W.\leqno{(\clubsuit)}$$
An easy computation of the subspaces $(E_i^* \otimes E_j)^W$ based on the previous remarks show that 
$$(1,\euler_0,\euler_0^{(2)},\dots,\euler_0^{(d-1)},(X^d-Y^d)(x^d-y^d),\ab_{1,0},\ab_{2,0},\dots,\ab_{d-1,0})$$
is a $\CM$-basis of $(E \otimes E^*)^W$. By $(\clubsuit)$, it is also a $P_\bullet$-basis of 
$\CM[V \times V^*]^W$. On the other hand, 
$$(X^d-Y^d)(x^d-y^d)=2 \euler_0^{(d)} - Rr,$$
so $(1,\euler_0,\euler_0^{(2)},\dots,\euler_0^{(d)},\ab_{1,0},\ab_{2,0},\dots,\ab_{d-1,0})$ is a 
$P_\bullet$-basis of $\CM[V \times V^*]^W$. By~(\ref{eq:changement-base-euler-diedral}),
$$
\CM[V\times V^*]^W = P_\bullet \oplus 
P_\bullet \,\euler_0 \oplus P_\bullet \,\euler_0^2 \oplus \cdots \oplus P_\bullet \,\euler_0^d 
\oplus P_\bullet \,\ab_{1,0} \oplus P_\bullet \,\ab_{2,0} \oplus \cdots \oplus P_\bullet \,\ab_{d-1,0},
$$
which shows the last assertion of the Theorem.

\medskip

It also proves that $\CM[V \times V^*]^W = \CM[q,r,Q,R,\euler_0,\ab_{1,0},\ab_{2,0},\dots,\ab_{d-1,0}]$.
Let $A_1$, $A_2$, \dots, $A_{d-1}$ be indeterminates over $\CM[q,r,Q,R]$: 
we have a surjective morphism 
$$\CM[q,r,Q,R,E,A_1,A_2,\dots,A_{d-1}] \longsurto \CM[V \times V^*]^W$$
which sends 
$q$, $r$, $Q$, $R$, $E$, $A_1$, $A_2$, \dots, $A_{d-1}$ on 
$q$, $r$, $Q$, $R$, $\euler_0$, $\ab_{1,0}$, $\ab_{2,0}$,\dots, $\ab_{d-1,0}$ respectively.

For $1 \le i \le j \le d-1$, let $F_i$ (respectively $F_{i,j}$) denote the element of 
the polynomial algebra 
$\CM[q,r,Q,R,E,A_1,A_2,\dots,A_{d-1}]$ corresponding to the relation $(Z_i^0)$ (respectively 
$(Z_{i,j}^0)$). Let $A$ denote the quotient of $\CM[q,r,Q,R,E,A_1,A_2,\dots,A_{d-1}]$ 
by the ideal $\aG$ generated by the $F_i$'s and the $F_{i,j}$'s. 
We denote by 
$q_0$, $r_0$, $Q_0$, $R_0$, $E_0$, $A_{1,0}$, $A_{2,0}$, \dots, $A_{d-1,0}$ 
the respective images of $q$, $r$, $Q$, $R$, $E$, $A_1$, $A_2$, \dots, $A_{d-1}$ in $A$. We then 
have a surjective morphism of bigraded $\CM$-algebras $\ph : A \surto \CM[V \times V^*]^W$. 
We want to show that $\ph$ is an isomorphism. For this, it is sufficient to show that the 
bi-graded Hilbert series coincide. But, 
$$\dim_\CM^\bigrad(A) \ge \dim_\CM^\bigrad(\CM[V \times V^*]^W),\leqno{(\diamondsuit)}$$
where an inequality between two power series means that we have the corresponding inequality 
between all the coefficients. 

\medskip

We set 
$P_0=\CM[q_0,r_0,Q_0,R_0]$. Let 
$$A'=P_0 + P_0 E_0 + P_0 E_0^2 + \cdots + P_0 E_0^d + P_0 A_{1,0} + P_0 A_{2,0} + \cdots + P_0 A_{d-1,0}.$$
By construction, 
$$\dim_\CM^\bigrad(A') \le \dim_\CM^\bigrad(\CM[V \times V^*]^W).\leqno{(\heartsuit)}$$
We will prove that 
$$\text{$A'$ is a subalgebra of $A$.}\leqno{(\spadesuit)}$$
For this, taking into account the form of the $F_i$'s and the $F_{i,j}$'s, 
it is sufficient to show that $E_0^{d+1} \in A'$. But, by~(\ref{eq:changement-base-euler-diedral}),
$$A_{1,0}A_{d-1,0} = Q_0q_0 \Bigl(\sum_{0 \le j < (d-2)/2} n_{d-2,j} E_0^{d-2-2j}\Bigr) 
+ R_0r_0 - \sum_{0 \le j < d/2} n_{d,j} E_0^{d-2j},$$
and $n_{d,0}=1$. So 
$$E_0^d = -A_{1,0}A_{d-1,0} + Q_0q_0 \Bigl(\sum_{0 \le j < (d-2)/2} n_{d-2,j} E_0^{d-2-2j}\Bigr) 
+ R_0r_0 - \sum_{1 \le j < d/2} n_{d,j} E_0^{d-2j}$$
and so 
$$E_0^{d+1} = - E_0 A_{1,0}A_{d-1,0} + Q_0q_0 \Bigl(\sum_{0 \le j < (d-2)/2} n_{d-2,j} E_0^{d-1-2j}\Bigr) 
+ R_0r_0 - \sum_{1 \le j < d/2} n_{d,j} E_0^{d+1-2j}.$$
It is then sufficient to show that $E_0 A_{1,0} A_{d-1,0} \in A'$. But 
$E_0 A_{1,0} A_{d-1,0} = (q_0 A_{2,0} + Q_0 r_0) A_{d-1,0}$, which concludes the proof of $(\spadesuit)$.

\medskip

Since $A'$ contains $q_0$, $r_0$, $Q_0$, $R_0$, $E_0$, $A_{1,0}$, $A_{2,0}$, \dots, $A_{d-1,0}$ 
and since $A$ is generated by these elements, we have $A=A'$. It follows from $(\diamondsuit)$ 
and $(\heartsuit)$ that $\dim_\CM^\bigrad(A) = \dim_\CM^\bigrad(\CM[V \times V^*]^W)$, 
which shows that $\ph : A \to \CM[V \times V^*]^W$ is an isomorphism. In other words, 
this shows that the presentation of $\CM[V \times V^*]^W$ given in 
Theorem~\ref{theo:diedral-invariants} is correct. 

\medskip

It remains to prove the minimality of this presentation. The minimality of the number 
of generators follows from~\cite{bonnafe thiel}. Let us now prove the minimality of 
the number of relations. 
For this, let $\pG_0$ denote the bi-graded maximal ideal of $P_\bullet$ and set 
$B=\CM[V \times V^*]^W/\pG_0 \CM[V \times V^*]^W$. 
We denote by $e$, $a_1$, $a_2$,\dots, $a_{d-1}$ the respective images of 
$\euler_0$, $\ab_{1,0}$, $\ab_{2,0}$,\dots, $\ab_{d-1,0}$ in $B$. Hence,
$$B=\CM \oplus \CM e \oplus \CM e^2 \oplus \cdots \oplus \CM e^d 
\oplus \CM a_1 \oplus \CM a_2 \oplus \cdots \oplus \CM a_{d-1}$$
and $B$ admits the following presentation:
$$
\begin{cases}
\text{Generators:} & e,a_{1},a_{2},\dots,a_{d-1} \\
\text{Relations ($1 \le i \le j \le d-1$)~:} & 
\begin{cases}
ea_i=0 \\
a_ia_j =
\begin{cases}
0 & \text{si $i+j \neq d$,}\\
-e^d & \text{si $i+j=d$.}
\end{cases}
\end{cases}
\end{cases}
$$
It is sufficient to prove that the number of relations of $B$ is minimal. 
By reducing modulo the ideal $Be$, we get that all the relations of the form 
$a_ia_j = 0$ or $-e^d$ are necessary. By reducing modulo 
the ideal $B a_1 + \cdots + B a_{i-1} + B a_{i+1} + \cdots B a_{d-1}$, 
we get that the relations $ea_i=0$ are necessary. 
\end{proof}

\bigskip

\begin{rema}\label{rem:tau-invariants}
It is easily checked that the element $\t$ defined in~\S\ref{sub:diedral} satisfies
$$\lexp{\t}{q}=q,\quad \lexp{\t}{Q}=Q,\quad \lexp{\t}{\euler_0}=\euler_0
\qquad \text{and}\qquad \lexp{\t}{\ab_{i,0}}=-\ab_{i,0}$$
for all $0 \le i \le d$ (this follows from the fact that $\xi^d=-1$).\finl
\end{rema}

\bigskip

\begin{rema}[Gradings]\label{rem:grading}
The algebra $\CM[V \times V^*]$ admits a natural $(\NM \times \NM)$-grading, by putting elements of 
$V$ in bi-degree $(1,0)$ and elements of $V^*$ in bi-degree $(0,1)$. This bi-grading is stable 
under the action of $W$, so $\CM[V \times V^*]^W$ inherits this bi-grading. 
Note that the generators and the relations given by Theorem~\ref{theo:diedral-invariants} 
are bi-homogeneous.

This $(\NM \times \NM)$-grading induces a $\ZM$-grading such that any bi-homogeneous 
element of bi-degree $(m,n)$ is $\ZM$-homogeneous of degree $n-m$ (in other words, 
elements of $V$ have $\ZM$-degree $-1$ while elements of $V^*$ have $\ZM$-degree $1$).\finl
\end{rema}

\bigskip

\section{Cherednik algebras}

\medskip

\subsection{Definition}
We denote by $\CCB$ the $\CM$-vector space of maps $\Ref(W) \to \CM$ which are constant 
on conjugacy classes. If $i \in \ZM/d\ZM$, we denote by $C_i$ the element of 
$\CCB^*$ which sends $c \in \CCB$ to $c_{s_i}$. By~(\ref{eq:c}), $C_i=C_{i+2}$. 
Let $A=C_0$ and $B=C_1$. If $d$ is odd, then $A=B$ and $\CM[\CCB]=\CM[A]$ whereas, if $d$ is even, 
then $A \neq B$ (see~(\ref{eq:s-and-t})) and $\CM[\CCB]=\CM[A,B]$. 

The {\it generic rational Cherednik algebra at $t=0$} 
is the $\CM[\CCB]$-algebra $\Hb$ defined as the quotient 
of $\CM[\CCB] \otimes \bigl(\Trm(V \oplus V^*) \rtimes W\bigr)$ 
by the following relations (here, $\Trm(V \oplus V^*)$ 
is the tensor algebra of $V \oplus V^*$ over $\CM$):
\equat\label{relations-1}\begin{cases}
[u,u']=[U,U']=0, \\
\\
[u,U] = -2\DS{\sum_{i \in \ZM/d\ZM} \hskip1mm C_i 
\hskip1mm\frac{\langle u,\a_i \rangle \cdot \langle \a_i^\ve,U\rangle}{\langle \a_i^\ve,\a_i\rangle}
\hskip1mm s_i,} 
\end{cases}\endequat
for $U$, $U' \in V^*$ and $u$, $u' \in V$. Note that we have followed the 
convention of~\cite{calogero}.

% \bigskip
% 
% % \begin{rema}
% % Thanks to (\ref{action s V}), the second relation is equivalent to
% % \equat\label{eq:relations-1-sans-alpha}
% % [y,x] = T \langle y,x\rangle + \sum_{s \in \Ref(W)} C_s 
% % \langle s(y)-y,x\rangle 
% % \hskip1mm s.
% % \endequat
% % This avoids the use of $\a_s$ and $\a_s^\ve$.\finl
% % \end{rema}
% % 
% % \bigskip
% 
% \subsection{PBW Decomposition}\label{subsection:PBW-1} 
Given the relations~(\ref{relations-1}), 
the following assertions are clear:
\begin{itemize}
\item[$\bullet$] There is a unique morphism of $\CM$-algebras 
$\CM[V] \to \Hb$ sending $U \in V^* \subset \CM[V]$ to the class of
$U \in \Trm(V \oplus V^*) \rtimes W$ in $\Hb$.

\item[$\bullet$] There is a unique morphism of $\CM$-algebras 
$\CM[V^*] \to \Hb$ sending $u \in V \subset \CM[V^*]$ to the class of
$u \in \Trm(V \oplus V^*) \rtimes W$ in $\Hb$.

\item[$\bullet$] There is a unique morphism of $\CM$-algebras  
$\CM W \to \Hb$ sending $w \in W$ to the class of
$w \in \Trm(V \oplus V^*) \rtimes W$ in $\Hb$.

\item[$\bullet$] The $\CM$-linear map $\CM[\CCB] \otimes \CM[V] \otimes \CM W 
\otimes \CM[V^*] \longto \Hb$ induced by the three morphisms defined above
and the multiplication map is surjective.
Note that it is $\CM[\CCB]$-linear.
\end{itemize}
The last statement is strenghtened by the following fundamental 
result by Etingof and Ginzburg~\cite[Theorem 1.3]{EG} (see also~\cite[Theorem~4.1.2]{calogero}).

\bigskip

\begin{theo}[Etingof-Ginzburg]\label{PBW-1}
The multiplication map  $\CM[\CCB] \otimes \CM[V] \otimes \CM W 
\otimes \CM[V^*] \longto \Hb$ is an isomorphism of $\CM[\CCB]$-modules.
\end{theo}

\bigskip

\begin{rema}\label{rem:tau-parametres}
By~\cite[\S{3.5.C}]{calogero}, the group $\Nrm_{\Gb\Lb_\CM(V)}(W)$ acts 
naturally on $\Hb$. It follows from~(\ref{eq:tau}) that
$$\lexp{\t}{A}=B\qquad\text{and}\qquad \lexp{\t}{B}=A.$$
Here, $\t$ is the element of $\Nrm_{\Gb\Lb_\CM(V)}(W)$ defined 
in~\S\ref{sub:diedral}.\finl
\end{rema}

\bigskip

\begin{rema}[Gradings]\label{rem:grading-h}
The algebra $\Hb$ admits a natural $(\NM \times \NM)$-grading, by putting 
$V$ in bi-degree $(1,0)$, $V^*$ in bi-degree $(0,1)$, $W$ in degree $(0,0)$ and 
$\CCB^*$ in degree $(1,1)$ (see for instance~\cite[\S{3.2}]{calogero}). 

This $(\NM \times \NM)$-grading induces a $\ZM$-grading such that any bi-homogeneous 
element of bi-degree $(m,n)$ is $\ZM$-homogeneous of degree $n-m$. In other words, 
$\deg(V)=-1$, $\deg(V^*)=1$ and $\deg(W)=0=\deg(\CCB^*)=0$.\finl
\end{rema}

\bigskip

\subsection{Specialization}
Given $c \in \CCB$, we denote by $\CG_c$ the maximal ideal of $\CM[\CCB]$ defined by  
$\CG_c=\{f \in \CM[\CCB]~|~f(c)=0\}$: it is the ideal generated by 
$(C_i-c_{s_i})_{i \in \ZM/d\ZM}$. We set
$$\Hb_c = (\CM[\CCB]/\CG_c) \otimes_{\CM[\CCB]} \Hb = \Hb/\CG_c \Hb .$$
The $\CM$-algebra $\Hb_c$ is the quotient of the $\CM$-algebra 
$\Trm(V \oplus V^*) \rtimes W$ by the ideal generated by the
following relations: 
\equat\label{relations specialisees-0}\begin{cases}
[u,u']=[U,U']=0, \\
\\
[u,U] = -2\DS{\sum_{i \in \ZM/d\ZM} \hskip1mm c_{s_i} 
\hskip1mm\frac{\langle u,\a_i \rangle \cdot \langle \a_i^\ve,U\rangle}{\langle \a_i^\ve,\a_i\rangle}
\hskip1mm s_i,} 
\end{cases}\endequat
for $U$, $U' \in V^*$ and $u$, $u' \in V$. 

\bigskip

\begin{rema}[Grading]\label{rem:grading-hc}
The ideal $\CG_c$ is not bi-homogeneous (except if $c=0$) so the 
algebra $\Hb_c$ does not inherit from $\Hb$ an $(\NM \times \NM)$-grading. 
However, $\CG_c$ is $\ZM$-homogeneous, so $\Hb_c$ still admits a natural 
$\ZM$-grading.\finl
\end{rema}

\bigskip

\subsection{Calogero-Moser space}\label{sub:calogero}
We denote by $Z$ the centre of $\Hb$. By~\cite{EG}, it contains 
$\CM[V]^W$ and $\CM[V^*]^W$ so, by Theorem~\ref{PBW-1}, it contains 
the subalgebra
$$P = \CM[\CCB] \otimes \CM[V]^W \otimes \CM[V^*]^W.$$
Similarly, if $c \in \CCB$, we denote by $Z_c$ the centre of $\Hb_c$: 
it turns out~\cite[Corollary~4.2.7]{calogero} that $Z_c$ is the image of $Z$ and that 
the image of $P$ in $Z_c$ is $P_\bullet=\CM[V]^W \times \CM[V^*]^W$. 
Recall also (see for instance~\cite[Corollary~4.2.7]{calogero}) that 
\equat\label{eq:z-p-libre}
\text{\it $Z$ is a free $P$-module of rank $|W|$.}
\endequat

Since the $\CM$-algebra $\ZCB$ is finitely generated, we can associate to it 
an algebraic variety over $\CM$, called the {\it generic Calogero-Moser space}, 
and which will be denoted by $\ZCB$. If $c \in \CCB$, we denote by 
$\ZCB_{\! c}$ the algebraic variety associated with the $\CM$-algebra 
$Z_c$. 

\bigskip

\subsection{About the presentation of $Z$} 
We follow here the method of~\cite{bonnafe thiel}. 
If $h \in \Hb$, it follows from Theorem~\ref{PBW-1} that there exists 
a unique family of elements $(h_w)_{w \in W}$ of 
$\CM[\CCB] \otimes \CM[V] \otimes \CM[V^*]$ such that  
$$h=\sum_{w \in W} h_w w.$$
We define the $\CM[\CCB]$-linear map $\troncation : \Hb \to \CM[\CCB] \otimes \CM[V] \otimes \CM[V^*]$ 
by 
$$\troncation(h)=h_1.$$  
The next lemma is proved in~\cite{bonnafe thiel}:

\bigskip

\begin{lem}\label{lem:0}
The restriction of  $\troncation$ to $Z$ yields an isomorphism of bi-graded $\CM[\CCB]$-modules 
$$\troncation : Z \longiso \CM[\CCB \times V \times V^*]^W.$$
% which satisfies
% $$\troncation(z)-z \in \CG_0 \Hb.$$
% In other words, if $f \in \kb[\CCB \times V \times V^*]^W \subset \kb[\CCB] \otimes \kb[V] \otimes \kb[V^*]$, 
% then there exists a unique element $z=\sum_{w \in W} z_w w$ of $Z$ (where 
% $z_w \in \kb[\CCB] \otimes \kb[V] \otimes \kb[V^*]$) such that 
% $$\begin{cases}
% z_1 = f,\\
% z_w \equiv 0 \mod \CG_0 \otimes \kb[V] \otimes \kb[V^*] & \text{if $w \neq 1$}.
% \end{cases}$$
\end{lem}

\bigskip

We then set $\euler=\troncation^{-1}(\euler_0)$ and, for $0 \le i \le d$, 
$$\ab_i=\troncation^{-1}(\ab_{i,0}).$$
An explicit algorithm for computing the inverse map $\troncation^{-1}$ is described in~\cite{bonnafe thiel}. 
Note that $\troncation^{-1}(p)=p$ for $p \in P$, so that $\ab_0=\ab_{0,0}=r$ and 
$\ab_d=\ab_{d,0}=R$. By~\cite{bonnafe thiel}, 
the relations $(\Zrm_i^0)_{1 \le i \le d-1}$ and $(\Zrm_{i,j}^0)_{1 \le i \le j \le d-1}$ 
can be deformed into relations 
$(\Zrm_i)_{1 \le i \le d-1}$ and $(\Zrm_{i,j})_{1 \le i \le j \le d-1}$ and 
it follows from Theorem~\ref{theo:diedral-invariants} and~\cite{bonnafe thiel} that:

\bigskip

\begin{theo}\label{theo:centre-diedral}
The centre $Z$ of $\Hb$ admits the following presentation, as a $\CM[\CCB]$-algebra:
$$
\begin{cases}
\text{\it Generators:} & q,r,Q,R,\euler,\ab_{1},\ab_{2},\dots,\ab_{d-1} \\
\text{\it Relations:} & 
\begin{cases}
(\Zrm_i) & \text{\it for $1 \le i \le d-1$,}\\
(\Zrm_{i,j}) & \text{\it for $1 \le i \le j \le d-1$.}\\
\end{cases}
\end{cases}
$$
This presentation is minimal, as well by the number of generators as by the number of relations 
(there are $d+4$ generators and $(d+2)(d-1)/2$ relations). Moreover,
$$
Z= P \oplus 
P \,\euler \oplus P \,\euler^2 \oplus \cdots \oplus P \,\euler^d 
\oplus P \,\ab_{1} \oplus P \,\ab_{2} \oplus \cdots \oplus P \,\ab_{d-1}.
$$
\end{theo}

\bigskip

It must be said that we have no way to determine explicitly the relations $(\Zrm_{i,j})$ 
in general: we will describe them precisely only for $d \in \{3,4,6\}$ 
in \S\ref{sec:exemples}. Note that the information provided by Theorem~\ref{theo:centre-diedral} 
is sufficient enough to be able to prove Theorem~\ref{theo:fixed-d} in Section~\ref{sec:fixed}.
% On the other hand, the relations $(\Zrm_i)$ are easily described, 
% as they are trivially deformed from the relations $(\Zrm_i^0)$:
% 
% \bigskip
% 
% \begin{prop}\label{prop:zi}
% If $1 \le i \le d-1$, then
% $$\euler~\ab_i = q\ab_{i+1} + Q \ab_{i-1}.\leqno{(\Zrm_i)}$$
% \end{prop}
% 
% \bigskip
% 
% \begin{proof}
% ...
% \end{proof}
% 

\bigskip

\begin{rema}[Gradings]\label{rem:grading-z}
The bi-grading and the $\ZM$-grading on the algebra $\Hb$ constructed in Remark~\ref{rem:grading-h} induce 
a bi-grading and a $\ZM$-grading on $Z$. Note that the map $\troncation$ is bi-graded, 
so that the generators given in Theorem~\ref{theo:centre-diedral} are bi-homogeneous. 

On the other hand, the deformation process for the relations described in~\cite{bonnafe thiel} 
respects the bi-grading. So we may, and we will, assume in the rest of this paper that 
the relations $(\Zrm_i)$ and $(\Zrm_{i,j})$ given in Theorem~\ref{theo:centre-diedral} 
are bi-homogeneous.\finl
\end{rema}

\bigskip

\boitegrise{\it From now on, and until the end of this paper, we fix a parameter 
$c \in \CCB$ and we set $a=c_s$ and $b=c_t$. 
Note that, if $d$ is odd, then $a=b$.}{0.75\textwidth}

\bigskip

\subsection{Poisson bracket}
Recall from~\cite[\S{4.4.A}]{calogero} that the algebra $Z$ is endowed with a $\CM[\CCB]$-linear 
Poisson bracket 
$$\{,\} : Z \times Z \longto Z,$$
which is a deformation of the Poisson bracket on $Z_0=\CM[V \times V^*]^W$ obtained by restriction 
of the $W$-equivariant canonical Poisson bracket on $\CM[V \times V^*]$. This Poisson 
bracket induces a Poisson bracket on $Z_c$. It satisfies 
\equat\label{eq:poisson-euler}
\{q,Q\}=\euler
\endequat
(see~\cite[\S{4}]{dezelee} or~\cite[\S{3}]{berest}).

\bigskip

\section{Calogero-Moser cellular characters}

\medskip
\def\gaudin{{\mathrm{Gau}}}

\def\cellchar{{\mathrm{CellChar}}}

The aim of this section is to determine, for all values of $c$, the 
Calogero-Moser $c$-cellular characters as defined in~\cite[\S{11.1}]{calogero}. 
It will be given in Table~\ref{table:cellulaires} at the end of this section. We will  
use the alternative definition~\cite[Theorem~13.4.2]{calogero}, which is more convenient for 
computational purposes (see also~\cite{bonnafe thiel}). So, following~\cite[Chapter~13]{calogero}, we set
$$D_x=\sum_{i \in \ZM/d\ZM} \e(s_i) c_{s_i} \frac{\langle x,\a_i \rangle}{\a_i} s_i =
- \sum_{i \in \ZM/d\ZM} c_{s_i} \frac{1}{X-\z^i Y} s_i
\in \CM(V)[W]$$
$$D_y=\sum_{i \in \ZM/d\ZM} \e(s_i) c_{s_i} \frac{\langle y,\a_i \rangle}{\a_i} s_i =
\sum_{i \in \ZM/d\ZM} c_{s_i} \frac{\z^i}{X-\z^i Y} s_i\in \CM(V)[W].
\leqno{\text{and}}$$
We denote by $\gaudin(W,c)$ the sub-$\CM(V)$-algebra of $\CM(V)[W]$ generated by $D_x$ and $D_y$ 
(it is commutative by~\cite[\S{13.4.B}]{calogero}). Note that this algebra is not necessarily split. 
If $L$ is a simple $\gaudin(W,c)$-module, and if $\chi \in \Irr(W)$, we denote by 
$\mult_{L,\chi}^\calo$ the multiplicity of $L$ in a composition series 
of the $\Res_{\gaudin(W,c)}^{\CM(V) W} \CM(V)E_\chi$, where $E_\chi$ is a $\CM W$-module 
affording the character $\chi$. We then set 
$$\g_L = \sum_{\chi \in \Irr(W)} \mult_{L,\chi}^\calo ~\chi.$$
The set of {\it Calogero-Moser $c$-cellular characters} is
$$\cellchar_c^\calo(W)=\{\g_L~|~L \in \Irr(\gaudin(W,c))\}.$$
We will denote by $\EC_1^\gaudin$ (respectively $\EC_\e^\gaudin$, respectively $\LC_k^\gaudin$) 
the restriction of $\CM(V)E_{\unb_W}$ (respectively $\CM(V)E_\e$, respectively $\CM(V)E_{\chi_k}$) 
to $\gaudin(W,c)$. If $d$ is even, then the restriction of $\CM(V)E_{\e_s}$ (respectively 
$\CM(V)E_{\e_t}$) to $\gaudin(W,c)$ will be denoted by $\EC_s^\gaudin$ (respectively $\EC_t^\gaudin$). 

\bigskip

\begin{rema}\label{rem:trouver-les-simples}
Note that, since $\gaudin(W,c) \subset \CM(V)W$, every simple $\gaudin(W,c)$-module 
occurs as a composition factor of some $\Res_{\gaudin(W,c)}^{\CM(V) W} \CM(V)E_\chi$, 
for $\chi$ running over $\Irr(W)$.\finl
\end{rema}

\bigskip

\begin{notation}\label{rem:dx-simple}
For simplifying the computation in this section, we set 
$$D_x'=\frac{Y^d-X^d}{d} D_x\qquad\text{and}\qquad D_y'=\frac{X^d-Y^d}{d} D_y,$$
so that $\gaudin(W,c)$ is the sub-$\CM(V)$-algebra of $\CM(V)W$ generated by $D_x'$ and $D_y'$.\finl
\end{notation}

\bigskip

\subsection{The case where $a=b=0$} 
Whenever $a=b=0$, there is only one Calogero-Moser $0$-cellular character~\cite[Corollary~17.2.3]{calogero}, 
namely the regular character $\sum_{\chi \in \Irr(W)} \chi(1)\chi$.

\bigskip

\subsection{The case where $a \neq b=0$} 
We will assume here, and only here, that $b=0\neq a$: this forces $d$ to be even 
(and we write $d=2e$). Let $W'$ be the subgroup of $W$ generated by $s=s_0$ and $s_2=tst$. Then 
$W'$ is a dihedral group of order $2e$ and
$$W=\langle t \rangle \ltimes W'.$$
Let $c'$ denote the restriction of $c$ to $\Ref(W')$: then $c'$ is constant (and equal to $a$) 
and $\gaudin(W,c)=\gaudin(W',c')$. 
It then follows from the definition that the Calogero-Moser $c$-cellular characters 
are all the characters of the form $\Ind_{W'}^W \g'$, where $\g'$ is a Calogero-Moser 
$c'$-cellular character of $W'$. These characters $\g'$ will be determined in the next subsection 
and so it follows that the list of $c$-cellular characters of $W$ is 
\equat\label{eq:cellular-b0}
1+\e_t,\quad \e_s+\e,\quad \sum_{k=1}^{(d-2)/2} \chi_k.
\endequat

\bigskip

\begin{rema}\label{rema:cellular-a0}
If $a=0$ and $b \neq 0$, then one can use the element $n \in \NC$ of order $2$ such that 
$\lexp{n}{s}=t$ and $\lexp{n}{t}=s$ to be sent back to the previous case. We then deduce 
from~(\ref{eq:cellular-b0}) that 
the list of $c$-cellular characters of $W$ is 
\equat\label{eq:cellular-a0}
1+\e_s,\quad \e_t+\e,\quad \sum_{k=1}^{(d-2)/2} \chi_k.
\endequat
Note that we have a semi-direct product decomposition $W=\langle s \rangle \ltimes \lexp{n}{W'}$.\finl
\end{rema}

\bigskip

\subsection{The equal parameters case} 
We assume here, and only here, that $a=b \neq 0$. Then, if $1 \le k \le d-1$, then 
$$\r_k(D_x')=-a \frac{Y^d-X^d}{d} \sum_{i \in \ZM/d\ZM}  \frac{1}{X-\z^i Y} 
\begin{pmatrix} 
0 & \z^{ki} \\
\z^{-ki} & 0 \\ 
\end{pmatrix} = a
\begin{pmatrix}
0 & X^{k-1} Y^{d-k} \\
X^{d-k-1} Y^{k} & 0 \\
\end{pmatrix}$$
by using~(\ref{eq:impair-2}). Similarly, 
$$\r_k(D_y')= a
\begin{pmatrix}
0 & X^{k} Y^{d-k-1} \\
X^{d-k} Y^{k-1} & 0 \\
\end{pmatrix}=\frac{X}{Y}\r_k(D_x').
$$
% Also,
% $$\r_0(D_x')=a
% \begin{pmatrix}
% 0 & X^{d-1} \\
% X^{d-1} & 0 \\
% \end{pmatrix}\qquad\text{and}\qquad \r_0(D_y')=a
% \begin{pmatrix}
% 0 &  Y^{d-1} \\
% Y^{d-1} & 0 \\
% \end{pmatrix} = \frac{Y^{d-1}}{X^{d-1}} \r_0(D_x').$$
If we denote by $M$ the diagonal matrix $\begin{pmatrix} X & 0 \\ 0 & Y \end{pmatrix}$, then 
it follows from the previous formulas that 
\equat\label{eq:rho-k-k+1}
\forall~1 \le k \le d-2,~\forall~D \in \gaudin(W,c),~M\r_k(D) M^{-1} = \r_{k+1}(D).
\endequat
This implies that 
\equat\label{eq:iso-k-k+1}
\forall~1 \le k \le d-2,~\LC_k^\gaudin \simeq \LC_{k+1}^\gaudin.
\endequat
Since $\Tr(\r_k(D_x'))=\Tr(\r_k(D_y'))=0$, the nature of the restriction 
of $\r_k$ to $\gaudin(W,c)$ depends on whether $-\det(\r_k(D_x'))=a^2X^dY^{d-2}$ is a 
square in $\CM(V)$. Two cases may occur:

\bigskip

\subsubsection*{First case: assume that $d$ is odd} 
Then $-\det(\r_k(D_x'))$ is not a square in $\CM(V)$ (for $1 \le k\le d-1$), so it follows 
that $\LC_k^\gaudin$ is simple (but not absolutely simple) and it follows from~(\ref{eq:iso-k-k+1}) 
\equat\label{eq:simple-gaudin-impair}
\Irr(\gaudin(W,c)) = \{\EC_1^\gaudin,\EC_{\e}^\gaudin, \LC_1^\gaudin\}.
\endequat
Moreover, the list of $c$-cellular characters is given in this case by
\equat\label{eq:cellular-impair}
\unb_W,\quad \e\quad\text{and}\quad \sum_{k=1}^{(d-1)/2} \chi_k.
\endequat

\bigskip

\subsubsection*{Second case: assume that $d$ is even} 
In this case, it is easily checked that $\EC_1^\gaudin$, $\EC_\e^\gaudin$, $\EC_s^\gaudin$ and 
$\EC_t^\gaudin$ are four non-isomorphic simple $\gaudin(W,c)$-modules. Also, if 
$1 \le k \le d-1$, then
$$\LC_k^\gaudin \simeq \EC_s^\gaudin \oplus \EC_t^\gaudin,$$
by~(\ref{eq:iso-k-k+1}). Therefore, 
\equat\label{eq:simple-gaudin-pair}
\Irr(\gaudin(W,c)) = \{\EC_1^\gaudin,\EC_{\e}^\gaudin, \EC_s^\gaudin,\EC_t^\gaudin\}.
\endequat
and the list of $c$-cellular characters is given in this case by
\equat\label{eq:cellular-pair}
\unb_W,\quad \e,\quad \e_s + \sum_{k=1}^{(d-2)/2} \chi_k \quad\text{and}\quad 
\e_t + \sum_{k=1}^{(d-2)/2} \chi_k.
\endequat

\bigskip

\subsection{The opposite parameters case} 
We assume here, and only here, that $b=-a \neq 0$. This forces $d$ to be even. Then, 
using the automorphism of $\Hb$ induced by the linear character $\e_s$ (see~\cite[\S{3.5.B}]{calogero}), 
one can pass from the equal parameter case to the opposite parameter case 
by tensorizing by $\e_s$. Therefore, 
the list of $c$-cellular characters is given in this case by
\equat\label{eq:cellular-pair-oppose}
\e_s,\quad \e_t,\quad \unb_W + \sum_{k=1}^{(d-2)/2} \chi_k \quad\text{and}\quad 
\e + \sum_{k=1}^{(d-2)/2} \chi_k.
\endequat

\bigskip

\subsection{The generic case} 
We assume here, and only here, that $ab(a^2-b^2) \neq 0$ (so that we are not in the cases 
covered by the previous subsections). Note that this forces $d$ to be even. We will 
prove that
the list of $c$-cellular characters is given in this case by
\equat\label{eq:cellular-pair-generic}
\unb_W, \quad \e,\quad \e_s,\quad \e_t\quad\text{and}\quad 
\sum_{k=1}^{(d-2)/2} \chi_k.
\endequat

\bigskip

\begin{proof}
We have, for $1 \le k \le e-1$,
\eqna
\r_k(D_x')&=&\DS{\frac{X^d-Y^d}{d}\Bigl( a \sum_{i \in \ZM/e\ZM} 
\frac{1}{X-\z^{2i}Y} \begin{pmatrix} 0 & \z^{2ki} \\ \z^{-2ki} & 0 \end{pmatrix} }\\
&& \DS{+ b \sum_{i \in \ZM/e\ZM} 
\frac{1}{X-\z^{2i+1}Y} \begin{pmatrix} 0 & \z^{k(2i+1)} \\ \z^{-k(2i+1)} & 0 \end{pmatrix}}\Bigr) .
\endeqna
So it follows from~(\ref{eq:pair-1}),~(\ref{eq:pair-2}) and~(\ref{eq:pair-3}) that
$$\r_k(D_x')=\frac{1}{2}
\begin{pmatrix}
0 & X^{k-1}Y^{e-k}\bigl((a-b)X^e + (a+b)Y^e\bigr) \\
X^{e-k-1}Y^k\bigl((a+b)X^e + (a-b) Y^e\bigr) & 0 \\
\end{pmatrix}.$$
The matrix $\r_k(D_y')$ can be computed similarly and we can deduce that,
$$\forall~1 \le k \le e-2,~\forall~D \in \gaudin(W,c),~M\r_k(D)M^{-1}=\r_{k+1}(D).$$
Therefore, 
$$\forall~1 \le k \le e-2, ~\LC_k^\gaudin \simeq \LC_{k+1}^\gaudin.\leqno{(*)}$$
Moreover,
$$-\det(\r_k(D_x'))=\frac{1}{4} X^{e-2}Y^e\bigl((a-b)X^e + (a+b)Y^e\bigr)
\bigl((a+b)X^e+(a-b)Y^e\bigr).$$
Since $ab(a^2-b^2) \neq 0$, $-\det(\r_k(D_x'))$ is not a square in $\CM(V)$, and so $\LC_k^\gaudin$ 
is simple (but not absolutely simple) for $1 \le k \le e-1$. 

Moreover, an easy computation shows that $\EC_1^\gaudin$, $\EC_\e^\gaudin$, $\EC_s^\gaudin$ 
and $\EC_t^\gaudin$ are pairwise non-isomorphic simple $\gaudin(W,c)$-modules. So it 
follows from $(*)$ that 
\equat\label{eq:simples-gaudin-generic}
\Irr(\gaudin(W,c))=\{\EC_1^\gaudin, \EC_\e^\gaudin, \EC_s^\gaudin, \EC_t^\gaudin, \LC_1^\gaudin\}
\endequat
and that~(\ref{eq:cellular-pair-generic}) holds.
\end{proof}

\bigskip

\subsection{Conclusion}
The following Table~\ref{table:cellulaires} gathers all the possible list of cellular characters of $W$, 
according to the values of the parameters $a$ and $b$.

\def\espace{\vphantom{$\DS{\sum_{\substack{A \\ B}}^{\substack{A \\ B}} \chi(1)\chi}$}}

\begin{table}\refstepcounter{theo}
\centerline{
\begin{tabular}{@{{{\vrule width 2pt}\,\,\,}}c@{{\,\,\,{\vrule width 1pt}\,\,\,}}c|c@{{\,\,\,{\vrule width 2pt}}}}
\hlinewd{2pt}
\petitespace
Parameters & $d=2e$ (even) & $d=2e-1$ (odd) 
\vphantom{$\DS{\frac{a}{A_{\DS{A}}}}$}\\
\hlinewd{1pt}
\espace $a=b=0$ & $\DS{\sum_{\chi \in \Irr(W)} \chi(1)\chi}$ & $\DS{\sum_{\chi \in \Irr(W)} \chi(1)\chi}$ \\
\hline
\espace $a\neq b = 0$ & $\unb_W+\e_t$, \hskip1mm $\e_s+\e$, \hskip1mm $\DS{\sum_{k=1}^{e-1} \chi_k}$& 
\diaghead{\hskip4cm}{}{}\\
\hline
\espace $a=0 \neq b $ & $\unb_W +\e_s$, \hskip1mm $\e_t+\e$, \hskip1mm $\DS{\sum_{k=1}^{e-1} \chi_k}$& 
\diaghead{\hskip4cm}{}{}\\
\hline
\espace $a=b \neq 0$ & $\unb_W$, \hskip1mm $\e$, 
\hskip1mm $\e_s+\DS{\sum_{k=1}^{e-1} \chi_k}$, \hskip1mm $\e_t+\DS{\sum_{k=1}^{e-1} \chi_k}$
& $\unb_W$, \hskip1mm $\e$, 
\hskip1mm $\DS{\sum_{k=1}^{e-1} \chi_k}$\\
\hline
\espace $a=-b\neq 0$& $\e_s$, \hskip1mm $\e_t$, 
\hskip1mm $\unb_W +\DS{\sum_{k=1}^{e-1} \chi_k}$, \hskip1mm $\e+\DS{\sum_{k=1}^{e-1} \chi_k}$& 
\diaghead{\hskip4cm}{}{}\\
\hline
\espace $ab(a^2-b^2) \neq 0$ & $\unb_W$, \hskip1mm $\e_s$, \hskip1mm $\e_t$, \hskip1mm $\e$, 
\hskip1mm $\DS{\sum_{k=1}^{e-1} \chi_k}$ & \diaghead{\hskip4cm}{}{} \\
\hlinewd{2pt}
\end{tabular}}

\bigskip

\caption{Calogero-Moser cellular characters of $W$}\label{table:cellulaires}
\end{table}

\begin{rema}\label{rem:conjecture-cellulaire}
Whenever $a$, $b \in \RM$, 
the Kazhdan-Lusztig cellular characters for the dihedral groups are easily computable 
(see for instance~\cite{lusztig}) and a comparison with 
Table~\ref{table:cellulaires} shows that they coincide with Calogero-Moser 
cellular characters: this is~\cite[Conjecture~CAR]{calogero} for dihedral groups.\finl
\end{rema}

\bigskip

\section{Calogero-Moser families}

\medskip

The aim of this section is to compute the {\it Calogero-Moser $c$-families} of $W$ 
(as defined in~\cite[\S{9.2}]{calogero}) for all values of $c$. The result is given in 
Table~\ref{table:familles}. 
Note that this result is not new: the Calogero-Moser families have been computed by Bellamy 
in his thesis~\cite{bellamy these}. We provide here a different proof, which uses the computation 
of Calogero-Moser cellular characters.

\bigskip

\subsection{Families}
To any irreducible character $\chi$, Gordon~\cite{gordon} associates a simple $\Hb_c$-module $\LC_c(\chi)$ 
(we follow the convention of~\cite[Proposition~9.1.3]{calogero}). We denote by $\O_\chi^c : Z \to \CM$ 
the morphism defined by the following property: if $z \in Z$, then $\O_\chi^c(z)$ is the 
scalar by which $z$ acts on $\LC_c(\chi)$ (by Schur's Lemma).
We say that two irreducible characters $\chi$ and $\chi'$ {\it belong to the same Calogero-Moser $c$-family} 
if $\O_\chi^c=\O_{\chi'}^c$ (see~\cite{gordon} or~\cite[Lemma~9.2.3]{calogero}: note that 
Calogero-Moser families are called {\it Calogero-Moser blocks} in~\cite{gordon}). 
We give here a different proof of a theorem of Bellamy~\cite{bellamy these}:

\bigskip

\begin{theo}[Bellamy]\label{theo:familles}
Let $c \in \CCB$ and let $\chi$ and $\chi'$ be two irreducible characters of $W$. Then 
$\chi$ and $\chi'$ lies in the same Calogero-Moser families if and only if 
$\O_\chi^c(\euler)=\O_{\chi'}^c(\euler)$. Consequently, the Calogero-Moser families 
are given by Table~\ref{table:familles}.
\end{theo}

\bigskip

\begin{proof}
By~\cite[Proposition~7.3.2]{calogero}, the values of $\O_\chi^c(\euler)$ are given as follows:
\begin{itemize}
\itemth{a} If $d=2e-1$ is odd, then 
$\begin{cases}
\O_{\unb_W}^c(\euler)=da,\\
\O_{\e}^c(\euler)=-da,\\
\O_{\chi_k}^c(\euler)=0 & \text{if $1 \le k \le e-1$.}
\end{cases}$

\vskip0.3cm

\itemth{b} If $d=2e$ is even, then 
$\begin{cases}
\O_{\unb_W}^c(\euler)=e(a+b),\\
\O_{\e}^c(\euler)=-e(a+b),\\
\O_{\e_s}^c(\euler)=e(a-b),\\
\O_{\e_t}^c(\euler)=e(b-a),\\
\O_{\chi_k}^c(\euler)=0 & \text{if $1 \le k \le e-1$.}
\end{cases}$
\end{itemize}
But two irreducible characters occuring in the same Calogero-Moser $c$-cellular 
character necessarily belong to the same Calogero-Moser $c$-family~\cite[Proposition~11.4.2]{calogero}. 
So the Theorem follows from (a), (b) and Table~\ref{table:cellulaires}.
\end{proof}

\bigskip

\begin{table}\refstepcounter{theo}
\centerline{
\begin{tabular}{@{{{\vrule width 2pt}\,\,\,}}c@{{\,\,\,{\vrule width 1pt}\,\,\,}}c|c@{{\,\,\,{\vrule width 2pt}}}}
\hlinewd{2pt}
\petitespace
Parameters & $d=2e$ (even) & $d=2e-1$ (odd) 
\vphantom{$\DS{\frac{a}{A_{\DS{A}}}}$}\\
\hlinewd{1pt}
\petitespace $a=b=0$ & $\Irr(W)$ & $\Irr(W)$ \\
\hline
\petitespace $a\neq b = 0$ & 
$\substack{\vphantom{\frac{A}{A_A}}\DS{\text{$\{\unb_W,\e_t\}$, \quad $\{\e_s,\e\}$,}} \\% \hskip1mm 
\vphantom{\frac{A^A}{A}}\DS{\text{$\{\chi_1,\dots,\chi_{e-1}\}$}}}$ & \diaghead{\hskip4.4cm}{}{}\\
\hline
\petitespace $a=0 \neq b $ & 
$\substack{\vphantom{\frac{A}{A_A}}\DS{\text{$\{\unb_W,\e_s\}$, \quad $\{\e_t,\e\}$,}} \\
\vphantom{\frac{A^A}{A}}\DS{\text{$\{\chi_1,\dots,\chi_{e-1}\}$}}}$ & \diaghead{\hskip4.4cm}{}{}\\
\hline
\petitespace $a=b \neq 0$ & 
$\vphantom{\substack{\DS{A} \\ \DS{A} \\ \DS{A} \\ \DS{A} \\ \DS{A} \\ \DS{A}}} 
\substack{\vphantom{\frac{A}{A_A}}\DS{\text{$\{\unb_W\}$, \quad $\{\e\}$,}} \\
\vphantom{\frac{A}{A_A}}\DS{\text{$\{\e_s,\e_t,\chi_1,\dots,\chi_{e-1}\}$}}}$
& $\substack{\vphantom{\frac{A}{A_A}}\DS{\text{$\{\unb_W\}$, \quad $\{\e\}$,}} \\
\vphantom{\frac{A}{A_A}}\DS{\text{$\{\chi_1,\dots,\chi_{e-1}\}$}}}$\\
\hline
\petitespace $a=-b\neq 0$& 
$\substack{\vphantom{\frac{A}{A_A}}\DS{\text{$\{\e_s\}$, \quad $\{\e_t\}$,}} \\ 
\vphantom{\frac{A^A}{A}}\DS{\text{$\{\unb_W,\e,\chi_1,\dots,\chi_{e-1}\}$}}}$ & \diaghead{\hskip4.4cm}{}{}\\
\hline
\petitespace $ab(a^2-b^2) \neq 0$ & 
$\substack{\vphantom{\frac{A}{A_A}}\DS{\text{$\{\unb_W\}$, \hskip1mm $\{\e_s\}$, 
\hskip1mm $\{\e_t\}$, \hskip1mm $\{\e\}$, }} \\
\vphantom{\frac{A^A}{A}}\DS{\text{$\{\chi_1,\dots,\chi_{e-1}\}$}}}$ & \diaghead{\hskip4.4cm}{}{}\\
\hlinewd{2pt}
\end{tabular}}

\bigskip

\caption{Calogero-Moser families of $W$}\label{table:familles}
\end{table}

\begin{rema}\label{rem:conjecture-familles}
Whenever $a$, $b \in \RM$, 
the Kazhdan-Lusztig families for the dihedral groups are easily computable 
(see for instance~\cite{lusztig}) and a comparison with 
Table~\ref{table:familles}: this is~\cite[Conjecture~FAM]{calogero} for dihedral groups. 
Note that this was already proved by Bellamy~\cite{bellamy these}.\finl
\end{rema}

\bigskip

\subsection{Cuspidal families}
Recall that the algebras $Z$ and $Z_c$ are endowed with a Poisson bracket $\{,\}$. 
This Poisson structure has been 
used by Bellamy~\cite{bellamy cuspidal} to define the notion of {\it cuspidal} Calogero-Moser families.
If $\FC$ is a Calogero-Moser $c$-family, we set $\mG_\FC^c=\Ker(\O_\chi^c) \subset Z_c$ 
(note that $\O_\chi^c$ factorizes through the projection $Z \surto Z_c$). The Calogero-Moser 
$c$-family $\FC$ is called {\it cuspidal} if $\{\mG_\FC^c,\mG_\FC^c \} \subset \mG_\FC^c$. 
They have been determined for most of the Coxeter groups by Bellamy and Thiel~\cite{bellamy thiel}. 
In our case, we recall here their result, as well as a proof for the sake of completeness.

\bigskip

\begin{prop}\label{prop:cuspidal}
The list of cuspidal Calogero-Moser families is given by Table~\ref{table:cuspidal}. The following 
properties hold:
\begin{itemize}
\itemth{a} A Calogero-Moser family $\FC$ is cuspidal if and only if $|\FC| \ge 2$ and $\chi_1 \in \FC$ 
(and then $\chi_k \in \FC$ for all $1 \le k < d/2$).

\itemth{b} There is at most one cuspidal family. If $d \ge 5$, there is always exactly one cuspidal family.
\end{itemize}

\end{prop}

\bigskip

\begin{proof}
The main (easy) observation is that $\{q,Q\}=\euler$. This implies 
that, if $\chi$ belongs to a cuspidal families, then $\O_\chi^c(\euler)=0$. 
Since it follows from Table~\ref{table:familles} that the Calogero-Moser $c$-families 
are determined by the values of $\O_\chi^c(\euler)$, this implies that 
there is at most one cuspidal Calogero-Moser $c$-family, and that it must contain 
$\chi_1$ (and $\chi_k$, for $1 \le k <d/2$).

Also, since a Calogero-Moser $c$-family of cardinality $1$ cannot be cuspidal~\cite{gordon}, 
this shows the ``only if'' part of (a). It remains to prove the ``if'' part of (a). 
So assume that $\chi_1 \in \FC$ and that $|\FC| \ge 2$. 
By Bellamy theory~\cite{bellamy cuspidal}, there exists a non-trivial 
parabolic subgroup $W'$ of $W$ and a cuspidal Calogero-Moser $c'$-family $\FC'$ 
of $W'$ (here, $c'$ denotes the restriction of $c$ to $\Ref(W')$) 
which are associated with $\FC$. Again, by~\cite{bellamy cuspidal}, 
$|\FC|=|\FC'|$. We must show that $W=W'$. So assume that $W' \neq W$. Then 
$|W'|=2$ and so $|\FC'| \le 2$ and $c'$ must be equal to $0$. 
This forces $|\FC| = 2$ and $ab=0$ (and $c \neq 0$). This can only 
occur in type $G_2$: but the explicit computation of the Poisson bracket 
in type $G_2$ shows that $\FC$ is necessarily cuspidal.
\end{proof}

\bigskip

\begin{table}\refstepcounter{theo}
\centerline{
\begin{tabular}{@{{{\vrule width 2pt}\,\,\,}}c@{{\,\,\,{\vrule width 1pt}\,\,\,}}c|c@{{\,\,\,{\vrule width 2pt}}}}
\hlinewd{2pt}
\petitespace
Parameters & $d=2e$ (even) & $d=2e-1$ (odd) 
\vphantom{$\DS{\frac{a}{A_{\DS{A}}}}$}\\
\hlinewd{1pt}
\petitespace $a=b=0$ & $\Irr(W)$ & $\Irr(W)$, $e \ge 2$ \\
\hline
\petitespace $a\neq b = 0$ & $\{\chi_1,\dots,\chi_{e-1}\}$, $e \ge 3$ & \diaghead{\hskip4.4cm}{}{}\\
\hline
\petitespace $a=0 \neq b $ & $\{\chi_1,\dots,\chi_{e-1}\}$, $e \ge 3$ & \diaghead{\hskip4.4cm}{}{}\\
\hline
\petitespace $a=b \neq 0$ & $\{\e_s,\e_t,\chi_1,\dots,\chi_{e-1}\}$, $e \ge 2$
& $\{\chi_1,\dots,\chi_{e-1}\}$, $e \ge 3$ \\
\hline
\petitespace $a=-b\neq 0$& $\{\unb_W,\e,\chi_1,\dots,\chi_{e-1}\}$, $e \ge 2$ & \diaghead{\hskip4.4cm}{}{}\\
\hline
\petitespace $ab(a^2-b^2) \neq 0$ & $\{\chi_1,\dots,\chi_{e-1}\}$, $e \ge 3$ & \diaghead{\hskip4.4cm}{}{}\\
\hlinewd{2pt}
\end{tabular}}

\bigskip

\caption{Calogero-Moser cuspidal families of $W$}\label{table:cuspidal}
\end{table}

\bigskip
\def\Lie{{{\LG\iG\eG}}}

%\subsection{Cotangent space} 

If $\FC$ is a cuspidal Calogero-Moser $c$-family, then the Poisson bracket $\{,\}$ 
stabilizes the maximal ideal $\mG_\FC^c$ and so it induces a Lie bracket $[,]$ 
on the cotangent space $\Lie_c(\FC)=\mG_\FC^c/(\mG_\FC^c)^2$. It is a question 
to determine in general the structure of this Lie algebra. We would like to emphasize here 
the following two particular intriguing examples (a proof will be given in~\S\ref{sec:exemples}, 
using explicit computations).

\bigskip

\begin{theo}\label{theo:lie}
Let $\FC$ be a cuspidal Calogero-Moser $c$-family of $W$. Then:
\begin{itemize}
\itemth{a} If $d=4$ (i.e. if $W$ is of type $B_2$) and $a=b \neq 0$, then 
$\Lie_c(\FC) \simeq \sG\lG_3(\CM)$ is a simple Lie algebra of type $A_2$.

\itemth{b} If $d=6$ (i.e. if $W$ is of type $G_2$) and $ab(a^2-b^2) \neq 0$, then 
$\Lie_c(\FC) \simeq \sG\pG_4(\CM)$ is a simple Lie algebra of type $B_2$.
\end{itemize}
\end{theo}

\bigskip

\bigskip

\section{Calogero-Moser cells}

\medskip

Let $c \in \CCB$. 
The main theme of~\cite{calogero} is a construction of partitions of $W$ into 
Calogero-Moser left, right and two-sided $c$-cells,
using a Galois closure $M$ of the field extension $\Frac(Z)/\Frac(P)$. Let $G$ denote 
the Galois group of the field extension $M/\Frac(P)$: the Calogero-Moser cells 
are defined~\cite[Definition~6.1.1]{calogero} as orbits of particular subgroups of $G$. 
Our aim in this section is to prove~\cite[Conjectures~L~and~LR]{calogero} whenever $d$ is odd. 
We first start by trying to determine the Galois group $G$.

\medskip

It is proved in~\cite[\S{5.1.C}]{calogero} 
that there is an embedding
$$G \longinjto \SG_W$$
(here, $\SG_W$ denotes the symmetric group on the set $W$, and we identify $G$ with its image) 
such that
$$\iota(W \times W) \subset G,$$
where $\iota : W \times W \longto \SG_W$ denotes the morphism 
obtained by letting $W \times W$ act by left and right translations 
($(x,y) \cdot z = xzy^{-1}$). Let $\AG_W$ denote the alternating group on $W$. 

\bigskip

\begin{theo}\label{theo:galois}
If $d$ is odd, then $G=\SG_W$.
\end{theo}

\bigskip

\begin{proof}
We first prove an easy lemma about finite permutation groups.

\bigskip

\begin{quotation}
\begin{lem}\label{lem:permutation}
Let $\G$ be a subgroup of $\SG_W$. We assume that:
\begin{itemize}
\itemth{1} $d$ is odd;

\itemth{2} $\G$ contains $\iota(W \times W)$;

\itemth{3} $\G$ is primitive.
\end{itemize}
Then $\G=\SG_W$.
\end{lem}

\begin{proof}[Proof of Lemma~\ref{lem:permutation}]
Since $d$ is odd, the action of $\s=\iota(c,c)$ on $W$ is a cycle of length $d$ 
(it fixes $\langle c\rangle$ and acts by a cycle on $W \setminus \langle c \rangle$ 
by~(\ref{eq:c})). Moreover, $\iota(c,1)$ and $\iota(1,c)$ belong to the centralizer of 
$\iota(c,c)$ in $\G$, so $C_\G(c) \neq \langle c \rangle$. 
Since $\G$ is primitive, it follows from~\cite[Exercise~7.4.12]{dixon} 
that $\G = \SG_W$ or $\AG_W$. 

But note that the action of $\iota(s,1)$ 
is a product of $d$ transpositions, so it is an odd permutation (because $d$ is odd). 
Therefore, $\G \not=\AG_W$ and the Lemma is proved.
\end{proof}
\end{quotation}

\bigskip

Assume that $d$ is odd. By the description of the Calogero-Moser $c$-families 
given in Table~\ref{table:familles}, it follows from~\cite[Theorem~10.2.7]{calogero} 
that there exists a subgroup $I$ of $G$ which have three orbits 
for its action on $W$, of respective lengths $1$, $1$ and $2d-2$. 
Since $\iota(W \times W)$ is transitive on $W$, $G$ is also transitive 
and we may assume that one of the two orbits of length $1$ is the singleton 
$\{1\}$. Let $\D W$ denotes the diagonal in $W \times W$. Its action 
on $W$ is by conjugacy: it has only one fixed point (because the center 
of $W$ is trivial). This proves that the subgroup $\langle I,\iota(\D W) \rangle$ 
acts transitively on $W \setminus \{1\}$. So $G$ is $2$-transitive and, 
in particular, primitive. The Theorem now follows from Lemma~\ref{lem:permutation} 
above.
\end{proof}

\bigskip

\begin{coro}\label{coro:cellules}
If $d$ is odd, then the Conjectures~\cite[Conjectures~LR~and~L]{calogero} hold.
\end{coro}

\bigskip

\begin{proof}
Assume that $c_s \in \RM$ for all $s \in \Ref(W)$. 
The computation of Calogero-Moser $c$-families and $c$-cellular characters shows that, 
if we choose randomly two prime ideals as in~\cite[Chapter~15]{calogero}, then the associated 
Calogero-Moser two-sided or left $c$-cells have the same sizes as the Kazhdan-Lusztig 
two-sided or left $c$-cells respectively (see~\cite[Chapters~10~and~11]{calogero}). Since the Galois 
group $G$ coincides with $\SG_W$, we can manage to change the prime ideals 
so that Calogero-Moser and Kazhdan-Lusztig $c$-cells coincide.
\end{proof}

\bigskip

\begin{rema}\label{rem:bet}
Let 
$$\SG_W^B=\{\s \in \SG_W~|~\forall~w \in W,~\s(w_0w)=w_0\s(w)\}$$
$$\SG_W^D=\SG_W^B \cap \AG_W.\leqno{\text{and}}$$
Note that, in our case, $\SG_W^B$ (respectively $\SG_W^D$) is a Weyl group 
of type $B_d$ (respectively $D_d$) 
and that $\SG_W^D$ is a normal subgroup of $\SG_W^B$ of index $2$.  

Assume here, and only here, that $d$ is {\it even}. It then follows 
from~\cite[Proposition~5.5.2]{calogero} 
that $G \subset \SG_W^B$. We would bet a few euros (but not more) that $G=\SG_W^D$. 
This has been checked for $d=4$ in~\cite[Theorem~19.6.1]{calogero} 
and it will be checked in~\ref{eq:galois-g2} whenever $d=6$. Let us just prove a few general facts.

First, let $\Wba=W/\Zrm(W)$ (it is a dihedral group of order $d$) and let 
$\bar{\iota} : \Wba \times \Wba \to \SG_\Wba$ denote the morphism induced by the 
action by left and right translations. Let $\Gba$ denote the image of $G$ in $\SG_\Wba$ 
(indeed $G \subset \SG_W^B$ and there is a natural morphism $\SG_W^B \to \SG_\Wba$).
Then, if $a=b \neq 0$, the Calogero-Moser two-sided $c$-cells 
have cardinalities $1$, $1$ and $2d-2$ (by~\cite[Theorem~10.2.7]{calogero}) 
so it follows from the definition of Calogero-Moser cells 
that there exists a subgroup $I_1$ of $G$ whose orbits have cardinalities $1$, $1$ and $2d-2$. 
Therefore, the image $\Iba_1$ of $I_1$ in $\Gba$ have orbits of cardinalities $1$ and $d-1$. 
Consequently, $\Gba$ is $2$-transitive. Similarly, taking $c$ such that 
$ab(a^2-b^2)\neq 0$, we get that there is a subgroup $\Iba_2$ of $\Gba$ whose orbits 
have cardinalities $1$, $1$ and $d-2$. Therefore,
$$\text{$\Gba$ is $3$-transitive.}\leqno{(\diamondsuit)}$$
On the other hand,
$$\text{$\Gba$ contains $\bar{\iota}(\Wba \times \Wba)$.}\leqno{(\heartsuit)}$$
As a consequence, we get 
$$\text{If $d/2$ is odd, then $\Gba=\SG_\Wba$.}\leqno{(\spadesuit)}$$
Indeed, this follows from Lemma~\ref{lem:permutation}.\finl
% We have also checked that $\SG_W^D \subset G \subset \SG_W^B$ for $d \in \{8,10,12,14,16,18,20\}$ 
% using {\tt GAP4}~\cite[??]{??}.\finl
\end{rema}

\bigskip

\bigskip

\section{Fixed points}\label{sec:fixed}

\medskip

The $\ZM$-grading on the $\CM$-algebra $\Hb$ (defined in Remark~\ref{rem:grading-h})  induces an 
action of the group $\CM^\times$ on $\Hb$ as follows~\cite[\S{3.5.A}]{calogero}. 
If $\xi \in \CM^\times$ then:
\begin{itemize}
\item[$\bullet$] If $y \in V$, then $\lexp{\xi}{y}=\xi^{-1}y$.

\item[$\bullet$] If $x \in V^*$, then $\lexp{\xi}{x}=\xi x$.

\item[$\bullet$] If $w \in W$, then $\lexp{\xi}{w}=w$.

\item[$\bullet$] If $s \in \Ref(W)$, then $\lexp{\xi}{C_s}=C_s$.
\end{itemize}
% This is equivalent as grading (over $\ZM$) the algebra $\Hb$ in such a way that 
% the elements of $V$ have degree $-1$, the elements of $V^*$ have degree $1$, the elements of 
% $W$ have degree $0$ and the elements of $\CCB$ have degree $0$.
So the center $Z$ inherits an action of $\CM^\times$, which may be viewed as a $\CM^\times$-action 
on the Calogero-Moser space $\ZCB$, which stabilizes all the fibers $\ZCB_c$ (for $c \in \CCB$). 

Now, if $m \in \NM^*$, we denote by $\mub_m$ the group of $m$-th root of unity 
in $\CM^\times$. In~\cite[Conjecture~FIX]{calogero}, R. Rouquier and the author conjecture that 
all the irreducible components of the fixed point variety $\ZCB^{\mub_m}$ are 
isomorphic to the Calogero-Moser space of some other complex reflection groups (here, 
$\ZCB^{\mub_m}$ is endowed with its reduced structure). 
This conjecture will be checked for $d \in \{3,4,6\}$ and any $m$ in Section~\ref{sec:exemples}. 

% In the particular case where there exists an element $w \in W$ admitting a regular 
% eigenvector for an eigenvalue which is a primitive $m$-th root of unity, then the above conjecture 
% is more precise~\cite[Conjecture~FIX]{calogero}: the irreducible component of maximal dimension of the 
% fixed point variety $\ZCB_c^{\mub_m}$ must be isomorphic 
% to the Calogero-Moser space of $C_W(w)$ associated with some parameter. We will prove 
% this conjecture in our special case for $m=d$, whenever $d$ is odd.

\bigskip

\begin{theo}\label{theo:fixed-d}
Assume that $d$ is odd. Then 
$$\ZCB^{\mub_d} \simeq \{(a,u,v,e) \in \CM^4~|~(e-da)(e+da)e^{d-2}=uv\}.$$
\end{theo}

\bigskip

\begin{rema}\label{rem:fixed}
By~\cite[Theorem~18.2.4]{calogero}, the above Theorem shows that $\ZCB_{\max,c}^{\mub_d}$ 
is isomorphic to the Calogero-Moser space associated with the cyclic group of order $d$ 
and some parameters, so it proves~\cite[Conjecture~FIX]{calogero} in this case.\finl
\end{rema}

\bigskip

\begin{proof}
The case where $d =1$ is not interesting, so we assume that $d \ge 3$. 
Let $\IG$ denote the ideal of $Z$ generated by $\{\lexp{\z}{z}-z~|~z \in Z\}$. 
Then the algebra of regular functions on $\ZCB^{\mub_d}$ is $Z/\sqrt{\IG}$. 
We will describe $Z/\IG$, and this will prove that $\IG=\sqrt{\IG}$ in this case. 
Therefore, 
$$\IG = \langle q,Q,\ab_1,\ab_2,\dots,\ab_{d-1}\rangle,$$
and so $Z/\IG$ is generated by the images of $A$, $r$, $R$ and $\euler$. 
In the quotient $Z/\IG$, all the equations of $\ZM$-degree which is not divisible by $d$ 
are automatically fulfiled, so it only remains the equations $(\Zrm_{i,d-i})$ 
(which is bi-homogeneous of bi-degree $(d,d)$). 
Also, note that $(\Zrm_{i,d-i}^0)$ implies that 
$$\euler^d = rR \mod \langle \IG, A \rangle.$$
The only bi-homogeneous monomials in $A$, $r$, $R$ and $\euler$ of bi-degree $(d,d)$ 
are $rR$ and the $\euler^k A^{d-k}$ (for $0 \le k \le d$). Therefore, the above equation 
implies that there exist complex numbers $\l_0$,\dots, $\l_{d-1}$ such that 
$$\euler^d + \l_{d-1} A\euler^{d-1} + \cdots + \l_1 A^{d-1}\euler + 
\l_0 A^d \equiv rR \mod \IG.\leqno{(*)}$$
On the other hand, it follows from~\cite[Corollary~9.4.4]{calogero} that
$$(\euler -dA)(\euler+dA)\euler^{2d-2}=\prod_{\chi \in \Irr(W)} (\euler - \O_\chi(\euler))^{\chi(1)^2} 
\equiv 0 \mod \langle q,Q,r,R \rangle.$$
So $(*)$ implies that the polynomial 
$\tb^d + \l_{d-1} A\tb^{d-1} + \cdots + \l_1 A^{d-1}\tb + \l_0 A^d$ divides 
$(\tb-dA)(\tb+dA)\tb^{2d-2}$ in $\CM[A][\tb]$. 

Since all the $\CM^\times$-fixed points belong to $\ZCB^{\mub_d}$, this implies that 
$\tb-dA$ and $\tb+dA$ both divide 
$\tb^d + \l_{d-1} A\tb^{d-1} + \cdots + \l_1 A^{d-1}\tb + \l_0 A^d$. 
Therefore, 
$$\tb^d + \l_{d-1} A\tb^{d-1} + \cdots + \l_1 A^{d-1}\tb + \l_0 A^d=
(\tb-dA)(\tb+dA)\tb^{d-2},$$
and so there remains only one relations in the quotient $Z/\IG$, namely
$$(\euler-dA)(\euler+dA)\euler^{d-2} \equiv rR \mod \IG,$$
as desired.
\end{proof}

\bigskip

\section{Examples}\label{sec:exemples}

\medskip

We are interested here in the cases where $d \in \{3,4,6\}$. This are the 
Weyl groups of rank $2$ (of type $A_2$, $B_2$ or $G_2$). For each of these cases, 
we give a complete presentation of the centre $Z$ of $\Hb$ (using the algorithms 
developed in~\cite{bonnafe thiel}). We use these explicit computations 
to check some of the facts that have been stated earlier in this paper. 
Most of the computations are done using {\tt MAGMA}~\cite{magma}.

\bigskip

\subsection{The type $A_2$}
We work here under the following hypothesis:

\medskip

\boitegrise{\it We assume in this subsection, and only in this subsection, that $d=3$. 
In other words, $W$ is a Weyl group of type $A_2$.}{0.75\textwidth}

\bigskip

Using {\tt MAGMA} and~\cite{bonnafe thiel}, we can compute 
effectively the generators of the $\CM[\CCB]$-algebra $Z$ and 
we obtain the following presentation for $Z$ (note that $\CM[\CCB]=\CM[A]$ 
because $A=B$):

\bigskip

\begin{prop}\label{prop:centre-a2}
The $\CM[A]$-algebra $Z$ admits the following presentation:
$$\begin{cases}
\text{Generators:} & q,r,Q,R,\euler, \ab_1,\ab_2 \\
\text{Relations:} &
\begin{cases} 
\euler~\ab_1 = q ~\ab_2 + rQ\\
\euler~\ab_2 = Q~\ab_1 + qR\\
\ab_1^2 = 4q^2Q+r ~\ab_2 - q~\euler^2 + 9A^2~q\\
\ab_1~\ab_2 = 4qQ~\euler + rR - \euler^3 + 9A^2~\euler\\
\ab_2^2 = 4qQ^2  + R~\ab_1 - Q~\euler^2+ 9A^2~Q\\
\end{cases}
\end{cases}$$
\end{prop}

% \begin{verbatim}
%   eu*a[1] eq q*a[2]+r*Q,
%   eu*a[2] eq Q*a[1]+q*R,
%   a[1]^2 eq 4*q^2*Q+r*a[2]-q*eu^2+9*A^2*q,
%   a[1]*a[2] eq 4*q*Q*eu + r *R - eu^3 + 9*A^2*eu,
%   a[2]^2 eq 4*q*Q^2 + R *a[1] - Q*eu^2 + 9*A^2*Q
% \end{verbatim}
% 

\bigskip

The minimal polynomial of $\euler$ is given by
\equat\label{eq:minimal-euler-a2}
\begin{split}
\tb^6 - (6qQ+9A^2)~\tb^4 -rR~\tb^3 + 9(q^2Q^2+2A^2qQ)~\tb^2 \hskip1cm\\
~\hskip1cm + 3qrQR~\tb  + q^3R^2 + r^2Q^3 - 4q^3Q^3-9A^2q^2Q^2.
\end{split}
\endequat

\bigskip

We conclude by proving~\cite[Conjecture~FIX]{calogero} in this case (about the variety 
$\ZCB^{\mub_m}$). Note that the only interesting case is where $m$ divides the order of an element 
of $W$. So $m \in \{1,2,3\}$. The case $m=1$ is stupid while the case $m=3$ is treated 
in Theorem~\ref{theo:fixed-d}:

\bigskip

\begin{prop}\label{prop:a2-mu2}
The $\CM[A]$-algebra $\CM[\ZCB^{\mub_2}]$ admits the following presentation:
$$\begin{cases}
\text{Generators:} & q,Q,\eulerq \\
\text{Relations:} &
\begin{cases} 
q(\eulerq^2 - 9A^2 - 4qQ)=0,\\
\eulerq(\eulerq^2 - 9A^2 - 4qQ)=0,\\
Q(\eulerq^2 - 9A^2 - 4qQ)=0,\\
\end{cases}
\end{cases}$$
In particular, if $a \neq 0$, then the variety $\ZCB_c^{\mub_2}$ has two irreducible 
components:
\begin{itemize}
\itemth{1} A component of dimension $2$ defined by the equation $\eulerq^2 - 9a^2 - 4qQ=0$ 
(which contains the points $z_1$ and $z_\e$).

\itemth{2} An isolated point, which is equal to $z_{\chi_1}$.
\end{itemize}
\end{prop}

\bigskip

% \begin{prop}\label{prop:a2-mu3}
% The $\CM[A]$-algebra $\CM[\ZCB^{\mub_3}]$ admits the following presentation:
% $$\begin{cases}
% \text{Generators:} & r,R,\eulerq \\
% \text{Relation:} &
% \eulerq^3 - 9A^2 \eulerq = rR,\\
% \end{cases}$$
% In particular, the variety $\ZCB_c^{\mub_2}$ is isomorphic 
% to the Calogero-Moser space associated with the cyclic group 
% of order $3$ and to parameters $(0,3a,-3a)$...
% \end{prop}
% 
% 
% 
% \bigskip

\subsection{The type $B_2$}
We work here under the following hypothesis:

\medskip

\boitegrise{\it We assume in this subsection, and only in this subsection, that $d=4$. 
In other words, $W$ is a Weyl group of type $B_2$.}{0.75\textwidth}

\bigskip

Using {\tt MAGMA} and~\cite{bonnafe thiel}, we can compute 
effectively the generators of the $\CM[\CCB]$-algebra $Z$ and 
we obtain the following presentation for $Z$ (note that $A \neq B$):

\bigskip

\begin{prop}\label{prop:centre-b2}
The $\CM[A,B]$-algebra $Z$ admits the following presentation:
$$\begin{cases}
\text{Generators:} & q,r,Q,R,\euler, \ab_1,\ab_2,\ab_3 \\
\text{Relations:} &
\begin{cases} 
\euler~\ab_1 &\hskip-0.4cm= q ~\ab_2 + rQ-2(A^2-B^2)~q\\
\euler~\ab_2 &\hskip-0.4cm= q~\ab_3 + Q~\ab_1-2(A^2-B^2)~\euler\\
\euler~\ab_3 &\hskip-0.4cm= qR + Q~\ab_2-2(A^2-B^2)~Q\\
\ab_1^2 &\hskip-0.4cm= r~\ab_2-q^2~\euler^2+4q^3Q
+2(A^2-B^2)~r+8(A^2+B^2)~q^2\\
\ab_1~\ab_2 &\hskip-0.4cm= r~\ab_3 - q~\euler^3 + 4q^2Q~\euler
+2(A^2-B^2)~\ab_1+8(A^2+B^2)~q~\euler\\
\ab_1~\ab_3 &\hskip-0.4cm= rR -\euler^4 + 5qQ \euler^2 - 4q^2Q^2+4(A^2-B^2)~\ab_2\\
&
+8(A^2+B^2)~\euler^2 - 8(A^2+B^2)~qQ
-8(A^2-B^2)^2\\
\ab_2^2 &\hskip-0.4cm= rR - \euler^4 + 4qQ~\euler^2+4(A^2-B^2)~\ab_2\\
&
+8(A^2+B^2)~\euler^2 
-4(A^2-B^2)^2\\
\ab_2~\ab_3 &\hskip-0.4cm= R~\ab_1 - Q~\euler^3 + 4qQ^2~\euler
+2(A^2-B^2)~\ab_3+8(A^2+B^2)~Q~\euler\\
\ab_3^2 &\hskip-0.4cm= R~\ab_2 - Q^2~\euler^2 + 4qQ^3+2(A^2-B^2)~R+8(A^2+B^2)~Q^2\\
\end{cases}
\end{cases}$$
\end{prop}

\bigskip

% \begin{verbatim}
%   eu*a[1] - (q*a[2]+r*Q-2*(A^2-B^2)*q),
%   eu*a[2] - (q*a[3]+Q*a[1]-2*(A^2-B^2)*eu),
%   eu*a[3] - (q*R+Q*a[2]-2*(A^2-B^2)*Q),
%   a[1]^2  - (r*a[2]-q^2*eu^2+4*q^3*Q+2*(A^2-B^2)*r+8*(A^2+B^2)*q^2),
%   a[1]*a[2]-(r*a[3]-q*eu^3+4*q^2*Q*eu+2*(A^2-B^2)*a[1]+8*(A^2+B^2)*q*eu),
%   a[1]*a[3]-(r*R-eu^4+5*q*Q*eu^2-4*q^2*Q^2+4*(A^2-B^2)*a[2]+8*(A^2+B^2)*eu^2-8*(A^2+B^2)*q*Q-8*(A^2-B^2)^2),
%   a[2]^2  - (r*R-eu^4+4*q*Q*eu^2+4*(A^2-B^2)*a[2]+8*(A^2+B^2)*eu^2-4*(A^2-B^2)^2),
%   a[2]*a[3]-(R*a[1]-Q*eu^3+4*q*Q^2*eu+2*(A^2-B^2)*a[3]+8*(A^2+B^2)*Q*eu),
%   a[3]^2  - (R*a[2]-Q^2*eu^2+4*q*Q^3+2*(A^2-B^2)*R+8*(A^2+B^2)*Q^2)
% \end{verbatim}
Now, let
\eqna 
f_4(\tb)&=&\tb^4-8(qQ+A^2+B^2)~\tb^3 + 
(20q^2Q^2-rR +32(A^2+B^2)qQ+16(A^2-B^2)^2)~\tb^2 \\
&& -4(4q^3Q^3-qrQR+2(A^2-B^2)(q^2R+rQ^2)+8(A^2+B^2)q^2Q^2)~\tb
+ (q^2R-rQ^2)^2.
\endeqna
Then
\equat
\text{\it The minimal polynomial of $\euler$ is $f_4(\tb^2)$.}
\endequat

% 
% It is easily checked that this presentation is equivalent to the presentation given 
% in~\cite[??]{calogero}, up to the following changes of variables (using the notation 
% in~\cite[??]{calogero}). 

\bigskip
\def\rim{\rG}
\def\qim{\qG}
\def\Rim{\RG}
\def\Qim{\QG}
\def\eulerim{\eG\uG}
\def\abim{\aG}

\subsubsection{Cuspidal point} 
We aim to prove here Theorem~\ref{theo:lie}(a). 
By Table~\ref{table:cuspidal}, there is a cuspidal point in 
$\ZCB_{\! c}$ if and only if $a^2=b^2$. Using the automorphism 
of $\Hb$ (and so, of $Z$) induced by the linear character $\e_s$ (see~\cite[\S{3.5.B}]{calogero}), 
we may reduce to the case where 
$$a=b.$$
Then there is only one cuspidal family $\FC$ in $\ZCB_{\! c}$ (the one containing $\chi_1$). 
It is easily checked that $\mG_\FC^c=\langle q,r,Q,R,\euler,\ab_1,\ab_2,\ab_3 \rangle_{Z_c}$ 
and it is readily checked 
from the presentation given in Proposition~\ref{prop:centre-b2} 
that the cotangent space $\Lie_c(\FC)=\mG_\FC^c/(\mG_\FC^c)^2$ has dimension $8$: so a basis 
is given by the images $\qim$, $\rim$, $\Qim$, $\Rim$, $\eulerim$, $\abim_1$, $\abim_2$, 
$\abim_3$ of $q$, $r$, $Q$, $R$, $\euler$, $\ab_1$, $\ab_2$, $\ab_3$ respectively. 
 
The computation of the Poisson bracket can be done using the {\tt MAGMA} package {\tt CHAMP}: writing 
the result modulo $(\mG_\FC^c)^2$ gives the Lie bracket on $\Lie_c(\FC)$. 
We can then deduce that:

\bigskip

\begin{prop}\label{prop:b2-lie}
Assume here that $a=b$. 
The linear map $\aleph_c$ defined in Table~\ref{table:sl3} 
is a morphism of Lie algebras. It is an isomorphism if 
$a \neq 0$.
\end{prop}

\begin{table}\refstepcounter{theo}
$$\begin{array}{ccc}
\qim & \longmapsto & \begin{pmatrix} 0 & 0 & 0 \\ 1 & 0 & 0  \\ 0 & 1 & 0 \end{pmatrix} \\
\rim  & \longmapsto & 8a \begin{pmatrix} 0 & 0 & 0 \\ 0 & 0 & 0 \\ 1 & 0 & 0 \end{pmatrix} \\
\Qim      & \longmapsto & \begin{pmatrix} 0 & -2 & 0 \\ 0 & 0 & -2 \\ 0 & 0 & 0  \end{pmatrix} \\
\Rim & \longmapsto & 32a \begin{pmatrix} 0 & 0 & 1 \\ 0 & 0 & 0\\ 0 & 0 & 0  \end{pmatrix} \\
\end{array}
\begin{array}{ccc}
\eulerim & \longmapsto & \begin{pmatrix} 2 & 0 & 0 \\ 0 & 0 & 0 \\ 0 & 0 & -2 \end{pmatrix} \\
\abim_2     & \longmapsto & \DS{\frac{8a}{3}} \begin{pmatrix} 1 & 0 & 0 \\ 0 & -2 & 0 \\ 0 & 0 & 1 \end{pmatrix} \\
\abim_1 & \longmapsto & 4a \begin{pmatrix} 0 & 0 & 0 \\ 1 & 0 & 0 \\ 0 & -1 & 0 \end{pmatrix} \\
\abim_3 & \longmapsto & 8a\begin{pmatrix} 0 & -1 & 0 \\ 0 & 0 & 1 \\ 0 & 0 & 0 \end{pmatrix} \\
\end{array}$$

\bigskip

\caption{Definition of $\aleph_c$ for $d=4$}\label{table:sl3}
\end{table}

\bigskip

\subsubsection{Fixed points}
The next result follows immediately from the presentation given in Proposition~\ref{prop:centre-b2}:

\bigskip

\begin{prop}\label{prop:fixed-b2}
The $\CM[A,B]$-algebra $\CM[\ZCB^{\mub_4}]$ admits the following presentation:
$$\begin{cases}
\text{Generators:} & r,R,\euler,\ab_2 \\
\text{Relations:} &
\begin{cases} 
\euler~(\ab_2+2(A^2-B^2) =0 \\
r~(\ab_2+2(A^2-B^2) =0\\
R~(\ab_2+2(A^2-B^2) =0\\
(\ab_2-2(A^2-B^2))(\ab_2+2(A^2-B^2))=0\\
(\euler-2(A+B))(\euler+2(A+B))(\euler-2(A-B))(\euler+2(A-B))\\ \qquad\qquad=rR  
+ 4(A^2-B^2)(\ab_2+2(A^2-B^2))\\
\end{cases}
\end{cases}$$
In particular:
\begin{itemize}
\itemth{a} If $a^2=b^2$, then $\ZCB_c^{\mub_4}$ is irreducible and is equal to
$$\{(e,u,v) \in \CM^3~|~(e-4a)(e+4a)e^2=uv\}.$$

\itemth{b} If $a^2 \neq b^2$, then $\ZCB_c^{\mub_4}$ has two irreducible 
components:
\begin{itemize}
\itemth{b1} The one of maximal dimension which is equal to
$$\{(e,u,v) \in \CM^3~|~(e-2(a+b))(e+2(a+b))(e-2(a-b))(e+2(a-b))=uv\},$$
and which contains $z_\unb$, $z_\e$, $z_{\e_s}$ and $z_{\e_t}$. 

\itemth{b2} An isolated point (corresponding to the maximal ideal 
$\langle r,R,\euler,\ab_2-2(a^2-b^2)\rangle$), which is equal to $z_{\chi_1}$.
\end{itemize}
\end{itemize}

\end{prop}

\bigskip

\subsection{The type $G_2$}
We work here under the following hypothesis:

\medskip

\boitegrise{\it We assume in this subsection, and only in this subsection, that $d=6$. 
In other words, $W$ is a Weyl group of type $G_2$.}{0.75\textwidth}

\bigskip

Using {\tt MAGMA} and~\cite{bonnafe thiel}, we can compute 
effectively the generators of the $\CM[\CCB]$-algebra $Z$ and 
we obtain the following presentation for $Z$ (note that $A \neq B$):

\bigskip

\begin{prop}\label{prop:centre-g2}
The $\CM[A,B]$-algebra $Z$ admits the following presentation:
$$\begin{cases}
\text{Generators:} &  r,R,\euler,\ab_1,\ab_2,\ab_3 \\
\text{Relations:} & \text{see Table~\ref{table:rel-g2}}\\
\end{cases}$$

\end{prop}

\bigskip

\begin{table}\refstepcounter{theo}
{\tiny $$\begin{cases}
\euler~\ab_1~=&q~\ab_2+r~Q-3~(A^2-B^2)~q^2\\
\euler~\ab_2~=&q~\ab_3+Q~\ab_1-3~(A^2-B^2)~q~\euler\\
\euler~\ab_3~=&q~\ab_4+Q~\ab_2-3~(A^2-B^2)~\euler^2+3~(A^2-B^2)~q~Q\\
\euler~\ab_4~=&q~\ab_5+Q~\ab_3-3~(A^2-B^2)~Q~\euler\\
\euler~\ab_5~=&q~R+Q~\ab_4-3~(A^2-B^2)~Q^2\\
\ab_1^2~=&r~\ab_2-q^4~\euler^2+4~q^5~Q+6~(A^2-B^2)~q~r+18~(A^2+B^2)~q^4\\
\ab_1~\ab_2~=&r~\ab_3-q^3~\euler^3+4~q^4~Q~\euler+3~(A^2-B^2)~r~\euler+18~(A^2+B^2)~q^3~\euler+3~(A^2-B^2)~q~\ab_1\\
\ab_1~\ab_3~=&r~\ab_4-q^2~\euler^4+5~q^3~Q~\euler^2-4~q^4~Q^2+18~(A^2+B^2)~q^2~\euler^2~\\
&~+6~(A^2-B^2)~q~\ab_2+3~(A^2-B^2)~Q~r-18~(A^2+B^2)~q^3~Q-18~(A^2-B^2)^2~q^2\\
\ab_1~\ab_4~=&r~\ab_5-q~\euler^5+6~q^2~Q~\euler^3-8~q^3~Q^2~\euler+18~(A^2+B^2)~q~\euler^3~\\
&~+9~(A^2-B^2)~q~\ab_3-36~(A^2+B^2)~q^2~Q~\euler+3~(A^2-B^2)~Q~\ab_1-54~(A^2-B^2)^2~q~\euler\\
\ab_1~\ab_5~=&r~R-\euler^6+7~q~Q~\euler^4-13~q^2~Q^2~\euler^2+4~q^3~Q^3+18~(A^2+B^2)~\euler^4+9~(A^2-B^2)~q~\ab_4\\
&-54~(A^2+B^2)~q~Q~\euler^2+9~(A^2-B^2)~Q~\ab_2+18~(A^2+B^2)~q^2~Q^2 \\
&-81~(A^2-B^2)^2~\euler^2+27~(A^2-B^2)^2~q~Q\\
\ab_2^2~=&r~\ab_4-q^2~\euler^4+4~q^3~Q~\euler^2+18~(A^2+B^2)~q^2~\euler^2+6~(A^2-B^2)~q~\ab_2\\
&+6~(A^2-B^2)~Q~r
-9~(A^2-B^2)^2~q^2\\
\ab_2~\ab_3~=&r~\ab_5-q~\euler^5+5~q^2~Q~\euler^3-4~q^3~Q^2~\euler+18~(A^2+B^2)~q~\euler^3+9~(A^2-B^2)~q~\ab_3~\\
&-18~(A^2+B^2)~q^2~Q~\euler+6~(A^2-B^2)~Q~\ab_1-36~(A^2-B^2)^2~q~\euler\\
\ab_2~\ab_4~=&r~R-\euler^6+6~q~Q~\euler^4-8~q^2~Q^2~\euler^2+18~(A^2+B^2)~\euler^4+12~(A^2-B^2)~q~\ab_4\\
&-36~(A^2+B^2)~q~Q~\euler^2+12~(A^2-B^2)~Q~\ab_2-81~(A^2-B^2)^2~\euler^2+18~(A^2-B^2)^2~q~Q\\
\ab_2~\ab_5~=&R~\ab_1-Q~\euler^5+6~q~Q^2~\euler^3-8~q^2~Q^3~\euler+3~(A^2-B^2)~q~\ab_5+18~(A^2+B^2)~Q~\euler^3\\
&+9~(A^2-B^2)~Q~\ab_3-36~(A^2+B^2)~q~Q^2~\euler-54~(A^2-B^2)^2~Q~\euler\\
\ab_3^2~=&r~R-\euler^6+6~q~Q~\euler^4-9~q^2~Q^2~\euler^2+4~q^3~Q^3+18~(A^2+B^2)~\euler^4-36~(A^2+B^2)~q~Q~\euler^2\\
&+12~(A^2-B^2)~q~\ab_4+12~(A^2-B^2)~Q~\ab_2+18~(A^2+B^2)~q^2~Q^2-72~(A^2-B^2)^2~\euler^2+36~(A^2-B^2)^2~q~Q\\
\ab_3~\ab_4~=&R~\ab_1-Q~\euler^5+5~q~Q^2~\euler^3-4~q^2~Q^3~\euler+6~(A^2-B^2)~q~\ab_5+18~(A^2+B^2)~Q~\euler^3\\
&+9~(A^2-B^2)~Q~\ab_3-18~(A^2+B^2)~q~Q^2~\euler-36~(A^2-B^2)^2~Q~\euler\\
\ab_3~\ab_5~=&R~\ab_2-Q^2~\euler^4+5~q~Q^3~\euler^2-4~q^2~Q^4+18~(A^2+B^2)~Q^2~\euler^2+6~(A^2-B^2)~Q~\ab_4\\
&+3~(A^2-B^2)~q~R-18~(A^2+B^2)~q~Q^3-18~(A^2-B^2)^2~Q^2\\
\ab_4^2~=&R~\ab_2-Q^2~\euler^4+4~Q^3~q~\euler^2+18~(A^2+B^2)~Q^2~\euler^2+6~(A^2-B^2)~Q~\ab_4\\
&+6~(A^2-B^2)~q~R
-9~(A^2-B^2)^2~Q^2\\
\ab_4~\ab_5~=&R~\ab_3-Q^3~\euler^3+4~q~Q^4~\euler+3~(A^2-B^2)~R~\euler+18~(A^2+B^2)~Q^3~\euler+3~(A^2-B^2)~Q~\ab_5\\
\ab_5^2~=&R~\ab_4-Q^4~\euler^2+4~q~Q^5+6~(A^2-B^2)~Q~R+18~(A^2+B^2)~Q^4
\end{cases}$$}

\caption{Presentation of $Z$ whenever $d=6$}\label{table:rel-g2}
\end{table}

Now, let
{\small\eqna
f_6(\tb)&=&\tb^6 -6\bigl(2qQ+3(A^2+B^2)\bigr)~\tb^5 + 9\bigl(6q^2Q^2+16(A^2+B^2)qQ+9(A^2-B^2)^2\bigr)~\tb^4
\\ && - \bigl(rR + 112q^3Q^3 +396(A^2+B^2)q^2Q^2 + 324(A^2-B^2)^2qQ\bigr)~\tb^3 \\ 
&& + 3\bigl(2rqRQ +35q^4Q^4- 6(A^2-B^2)(rQ^3+q^3R)  + 144(A^2+B^2)q^3Q^3 + 
162(A^2-B^2)^2q^2Q^2\bigr)~\tb^2 \\
&& -9\bigl(rq^2RQ^2+4q^5Q^5-4(A^2-B^2)(rqQ^4+q^4RQ) +18(A^2+B^2)q^4Q^4 +36(A^2-B^2)^2q^3Q^3\bigr)~\tb \\ 
&& + \bigl(rQ^3 + q^3R-9(A^2-B^2)q^2Q^2\bigr)^2
\endeqna}
Then 
\equat\label{eq:miniaml-g2}
\text{\it the minimal polynomial of $\euler$ is $f_6(\tb^2)$.}
\endequat

\subsubsection{Galois group}
Since $f_6(0)$ is a square in $P$, it follows from~\cite[(B.6.1)]{calogero} that 
the discriminant of $f_6(\tb^2)$ is a square in $P$. Therefore, the Galois group $G$ 
of the polynomial $f_6(\tb)$ is contained in $\AG_W$. Moreover, it follows from Remark~\ref{rem:bet} 
that $G$ is contained in $\SG_W^B$. Therefore, $G \subset \SG_W^D$. 
In fact,
\equat\label{eq:galois-g2}
G=\SG_W^D.
\endequat
\begin{proof}
Let $G_1$ denote the stabilizer of $1 \in W$ in $G \subset \SG_W$. 
By the computation of Calogero-Moser families, $G_1$ contains a subgroup admitting an orbit 
of cardinality $10$ (case $a=b \neq 0$). Consequently:
\begin{itemize}
\itemth{1} $G_1$ contains a subgroup admitting an orbit 
of cardinality $10$.

\itemth{2} $G$ contains $\iota(W \times W)$. 
\end{itemize}
An easy computation with the software {\tt GAP4}~\cite{gap} shows that 
the only subgroup $G$ of $\SG_W^D$ satisfying~(1) and~(2) 
is $\SG_W^D$.
\end{proof}

\bigskip

\subsubsection{Fixed points}
The next result follows immediately from the presentation given in Proposition~\ref{prop:centre-g2}:

\bigskip

\begin{prop}\label{prop:fixed-g2}
The $\CM[A,B]$-algebra $\CM[\ZCB^{\mub_6}]$ admits the following presentation:
$$\begin{cases}
\text{Generators:} & r,R,\euler \\
\text{Relation:} &
(\euler-3(A+B))(\euler+3(A+B))(\euler-3(A-B))(\euler+3(A-B))\euler^2=rR\\
\end{cases}$$
This proves~\cite[Conjecture~FIX]{calogero} in this case.
\end{prop}

\bigskip

\subsubsection{Lie algebra at cuspidal point}
Recall from Proposition~\ref{prop:cuspidal} that there is a unique cuspidal Calogero-Moser $c$-family $\FC$: 
it is the one which contains $\chi_1$ (this fact does not depend on the parameter $c$; 
however, the cardinality of $\FC$ depends on the parameter). It corresponds to the maximal 
ideal $\mG=\langle q,r,Q,R,\euler,\ab_1,\ab_2,\ab_3,\ab_4,\ab_5\rangle$ of $Z_c$. 
It follows from the presentation of $Z$ given by Proposition~\ref{prop:centre-g2} 
that the cotangent space $\Lie_c(\FC)=\mG_\FC^c/(\mG_\FC^c)^2$ has dimension $10$ (a basis is given by 
the images $\qim$, $\rim$, $\Qim$, $\Rim$, $\eulerim$, $\abim_1$, $\abim_2$, $\abim_3$, 
$\abim_4$ and $\abim_5$ of $q$, $r$, $Q$, $R$, $\euler$, $\ab_1$, $\ab_2$, $\ab_3$, $\ab_4$, 
$\ab_5$ in $\mG$ respectively). The Poisson bracket (and so the Lie bracket in $\Lie(\FC)$) 
can then be computed explicitly using the {\tt MAGMA} package {\tt CHAMP}. 
We can then deduce the following result:

\bigskip

\begin{prop}\label{prop:lie-g2}
Let $\aleph_c :  \Lie_c(\FC)  \longto  \sG\pG_4(\CM)$ be the linear map defined by 
Table~\ref{table:sp4}. It is a morphism of Lie algebras. Moreover:
\begin{itemize}
\itemth{a} If $a^2 \neq b^2$, then $\aleph_c$ is an isomorphism of Lie algebras.

\itemth{b} If $a^2=b^2$, then its image is isomorphic to $\sG\lG_2(\CM)$ (with basis 
$\aleph_c(\qim)$, $\aleph_c(\Qim)$ et $\aleph_c(\eulerim)$) and its kernel is commutative, 
of dimension $7$ (as a module for $\sG\lG_2(\CM)$, it is irreducible).
\end{itemize}
\end{prop}

\bigskip

\begin{table}\refstepcounter{theo}
{\small $$\begin{array}{ccc}
\rim & \longmapsto & 
-324(a^2-b^2)
\begin{pmatrix}
0 & 0 & 0 & 0 \\
0 & 0 & 0 & 0 \\
0 & 0 & 0 & 0 \\
1 & 0 & 0 & 0 
\end{pmatrix}\\
&&\\
\qim & \longmapsto & 
\begin{pmatrix}
 0 &  0 &  0 &   0 \\
-3 &  0 &  0 &   0 \\
 0 & -4 &  0 &   0 \\
 0 &  0 &  3 &   0 
 \end{pmatrix}\\
&&\\
\Rim & \longmapsto & 
9(a^2-b^2)
\begin{pmatrix}
0 & 0 & 0 & 1 \\
0 & 0 & 0 & 0 \\
0 & 0 & 0 & 0 \\
0 & 0 & 0 & 0 
\end{pmatrix}\\
&&\\
\Qim & \longmapsto & 
\begin{pmatrix}
0 & 1 & 0 & 0 \\
0 & 0 & 1 & 0 \\
0 & 0 & 0 & -1 \\
0 & 0 & 0 & 0 
\end{pmatrix}\\
&&\\
\eulerim & \longmapsto & 
\begin{pmatrix}
3 & 0 &  0 &  0 \\
0 & 1 &  0 &  0 \\
0 & 0 & -1 &  0 \\
0 & 0 &  0 & -3 
\end{pmatrix}\\
\end{array}\qquad
\begin{array}{ccc}
\ab_1 & \longmapsto & 
-54(a^2-b^2)
\begin{pmatrix}
0 & 0 & 0 & 0 \\
0 & 0 & 0 & 0 \\
1 & 0 & 0 & 0 \\
0 & 1 & 0 & 0 
\end{pmatrix}\\
&&\\
\ab_2 & \longmapsto & 
6(a^2-b^2)
\begin{pmatrix}
0 &  0 &  0 & 0 \\
3 &  0 &  0 & 0 \\
0 & -2 &  0 & 0 \\
0 &  0 & -3 & 0 
\end{pmatrix}\\
&&\\
\ab_3 & \longmapsto & 
3(a^2-b^2)
\begin{pmatrix}  
-3 & 0 &  0 & 0 \\
 0 & 2 &  0 & 0 \\
 0 & 0 & -2 & 0 \\
 0 & 0 &  0 & 3 
 \end{pmatrix}\\
&&\\
\ab_4 & \longmapsto & 
3(a^2-b^2)
\begin{pmatrix}
0 & -2 & 0 & 0 \\
0 &  0 & 1 & 0 \\
0 &  0 & 0 & 2 \\
0 &  0 & 0 & 0 
\end{pmatrix}\\
&&\\
\ab_5 & \longmapsto & 
(-9/2)(a^2-b^2)
\begin{pmatrix}
0 & 0 & 1 & 0 \\
0 & 0 & 0 & 1 \\
0 & 0 & 0 & 0 \\
0 & 0 & 0 & 0 
\end{pmatrix}\\
\end{array}$$}

\bigskip

\caption{Definition of $\aleph_c$ for $d=6$}\label{table:sp4}
\end{table}

\end{document}